\documentclass[11pt]{article}

\usepackage{epsfig,amsmath,latexsym,amssymb, amsthm, enumerate, mathtools}
\usepackage{graphicx}
\usepackage{lscape}
\usepackage{picture, eso-pic, tikz} %% for gray boxes
\usepackage{dsfont}
\usepackage{listings}
\usepackage{changes}
\usepackage{hyperref}
\usepackage{bbding}
\usepackage{booktabs}
\usepackage{comment}
\usepackage{multicol}
\usepackage{tikz}
\usepackage[ruled,vlined]{algorithm2e}

\usepackage{float}
\usepackage{caption}
\usepackage{pifont}

\renewcommand{\baselinestretch}{1.1}
\def\R{{\mathbb R}}  %%
\def\N{{\mathbb N}}  %%
\def\p{{\mathbb P}}  %% 
\def\E{{\mathbb E}}  %
\def\X{{\mathcal X}}  %%
\def\l{{\ell}}  %%
\def\D{{\mathcal D}}  %%
\def\T{{\mathcal T}}  %%
\def\B{{\mathcal B}}  %%
\def\DD{{\mathfrak D}}  %%
\def\TT{{\mathfrak T}}  %%
\def\BB{{\mathfrak B}}  %%
\def\LL{{\mathfrak L}}  %%
\def\Beweis{\footnotesize}
\DeclareMathOperator*{\argmin}{arg\,min}

\newcommand{\Remm}[1]{}
\newtheorem{theo}{Theorem}[section]

\newtheorem{prop}[theo]{Proposition}

\newtheorem{defi}[theo]{Definition}
\newtheorem{model ass}[theo]{Model Assumptions}

\newtheorem{example}[theo]{Example}

\newtheorem{rem}[theo]{Remark}
\def\EndProof{\hfill {\scriptsize $\Box$}}
\def\EndExample{\hfill {\scriptsize $\blacksquare$}}
\numberwithin{equation}{section}

\definecolor{MyGray}{rgb}{0.92,0.92,0.92}


\definecolor{British racing}{rgb}{0.0, 0.5, 0.0}
\def\bx{\boldsymbol{x}}
\def\bb{\boldsymbol{b}}

\def\bX{\boldsymbol{X}}

\def\b0{\boldsymbol{0}}

\def\bmu{\boldsymbol{\mu}}

\def\b0{\boldsymbol{0}}

\def\bb{\boldsymbol{b}}

\begin{document}
\author{Selim Gatti\footnote{RiskLab, Department of Mathematics, ETH Zurich,
selim.gatti@math.ethz.ch}}

%\date{Version of \today}
\date{Version of November 18, 2026}
\title{Assessing model calibration with boosting trees}
\maketitle

\begin{abstract} 
The main goal in regression modelling consists in approximating the conditional mean of a response given a set of features. A regression function is said to be strongly calibrated if the resulting mean estimates match the true conditional means for almost every set of features. Aiming for strong calibration seems not achievable in practice as one typically deals with finite samples of noisy observations. A weaker notion of calibration is auto-calibration, and it means that the conditional expectation
of the response, given a mean estimate, matches this estimate.  This notion is important, e.g., in insurance pricing as it ensures no cross-subsidization between different price cohorts. In this paper, we show that boosting trees can be used to test necessary conditions for strong calibration and auto-calibration, respectively. The practical relevance of our approach is supported by a numerical example, in which the proposed tests achieve high empirical power.

\vspace{0.5cm}
\noindent
%{\bf Keywords.} 
{\bf Keywords.} calibration, auto-calibration, regression modelling, mean estimation, insurance pricing.
\end{abstract}

\section{Introduction}

In regression modelling, the primary objective is to approximate the true conditional mean of a response given a set of features. To this end, various statistical models are used to fit a regression function that provides a mean estimate for each single set of features. This function is said to be {\it strongly calibrated} if the resulting mean estimates match the true conditional means for almost all features. Aiming for strong calibration seems not achievable in practice as models are fitted on finite samples of noisy observations. %Statistical tests for calibration can then be used to evaluate the accuracy of these models. Such tests have long been studied in the literature and include, for example, the Kolmogorov–Smirnov test introduced by Kolmogorov \cite{Kolmogorov} and Smirnov \cite{Smirnov}, which consists in evaluating the distance between the empirical distribution of the responses and the corresponding theoretical distribution under the null hypothesis of calibration. We refer to Chapter 16 of Lehmann--Romano \cite{Lehmann--Romano} for a deeper overview on calibration tests.
A weaker notion of calibration is {\it auto-calibration} (sometimes also called {\it mean-calibration}, {\it well-calibration} or {\it moderate calibration}); see, for example, Krüger--Ziegel \cite{Kruger}, Denuit et al.~\cite{Denuit2} and Van Calster et al.~\cite{Van Calster}. This notion goes back to earlier works on the reliability of probabilistic forecasts in meteorology; we refer to Bross \cite{Bross}, Sanders \cite{Sanders} and Murphy--Winkler \cite{Murphy--Winkler}.
%DeGroot--Fienberg \cite{DeGroot--Fienberg}. %An auto-calibrated regression function satisfies that the expected value of the responses being given the same mean estimate matches this estimate.
It means that when responses are grouped according to their mean estimates, the average of the responses within each group matches this estimate. This property is important in various applications where sums of mean estimates have to match sums of responses at a global and local level. This is, for example, the case in insurance pricing as an auto-calibrated pricing system avoids systematic cross-subsidy between different price cohorts; we refer the reader to Pohle \cite{Pohle}, Denuit et al.~\cite{Denuit}, Fissler et al.~\cite{Fissler} and Wüthrich--Merz \cite{WM2023}.

Testing for strong calibration has long been studied in the literature. For example, Bierens \cite{Bierens} considers integrated conditional moment tests, whereas Härdle--Mammen \cite{Härdle--Mammen} introduce a test based on the squared distance between the estimated regression function and a kernel-based approximation of the true conditional means. These approaches typically involve a complex construction and require several tuning choices. More recently, Delong et al.~\cite{Delong} propose a considerably simpler approach based on calibration bands for the special case of responses belonging to the exponential dispersion family.

In contrast to testing for strong calibration, the development of statistical tests for assessing auto-calibration has, surprisingly, received comparatively limited attention and has only emerged recently. Denuit et al.~\cite{Denuit2} propose a test for auto-calibration using Lorenz and concentration curves that requires the evaluation of a non-explicit asymptotic distribution using Monte-Carlo simulations. Building on this approach, Wüthrich \cite{Wüthrich} provides simpler versions of this test for discrete and finite regression functions. Additionally, Delong--Wüthrich \cite{Delong--Wüthrich} consider the use of bootstrap techniques and, for the special case of binary responses, Hosmer--Lemeshow \cite{Hosmer--Lemeshow} derive a $\chi^2$-test by binning observations over disjoint intervals, whereas Gneiting--Resin \cite{Gneiting} propose a bootstrap approach to test for auto-calibration in this binary setup. Moreover, Delong--Wüthrich \cite{Delong--Wüthrich2} consider the use of universal inference for responses belonging to the exponential dispersion family.

%while Delong et al.~\cite{Delong} propose the use of calibration bands to construct statistical tests for, both, calibration and auto-calibration.
%In the special case of binary observations, Hosmer--Lemeshow \cite{Hosmer--Lemeshow} derive a $\chi^2$-test by binning observations over disjoint intervals, whereas Gneiting--Resin \cite{Gneiting} propose a bootstrap approach to test for auto-calibration in this binary setup. 
%Assessing the auto-calibration of a regression function can be done graphically using {\it actual vs.~predicted} plots; see Wüthrich et al.~\cite{Wüthrich2}. These plots map the empirical mean of binned responses against their corresponding mean estimates. Moreover, s

We take a different approach in this paper. Our goal is to construct statistical tests for strong calibration and auto-calibration by only considering necessary conditions, which arise from the measure-theoretic definition of conditional expectation. In fact, such an approach has already been followed by Fissler et al.~\cite{Fissler}, who propose to perform joint Wald tests in order to test for a finite number of necessary conditions for strong calibration and auto-calibration, respectively. In practice, the use of their tests generally leads to low power as on the one hand, one needs multiple necessary conditions in order to be able to detect various kinds of violations of these properties, and on the other hand, the power of their tests is decreasing with an increasing number of necessary conditions. 

Our contribution is to show that boosting can be employed to test for strong calibration and auto-calibration by only assessing a single necessary condition, respectively, while being able to detect multiple kinds of violations of these two properties. Boosting is a class of ensemble methods that iteratively combines {\it weak learners}, i.e., simple models, to decrease the size of the residuals at each iteration. Its theoretical foundations were originally established by Valiant \cite{Valiant} and Kearns--Valiant \cite{Kearns--Valiant}. It was then popularized by Freund--Schapire \cite{Freund--Schapire} who came up with the AdaBoost algorithm for classification problems, and Friedman \cite{Friedman} later extended that framework to regression modelling using gradient-based methods. In the present work, we use trees as weak learners to detect violations of strong calibration and auto-calibration for a given regression function. These violations then enable us to derive a suitable necessary condition for each of the two properties above, which we assess through statistical tests with fully tractable asymptotic distributions. Through a numerical example, we show that although only a single necessary condition for strong calibration and auto-calibration, respectively, is assessed, the proposed testing procedures achieve a high power on a large motor third party liability insurance dataset.
%That is, although only one necessary condition for calibration and auto-calibration, respectively, is assessed, the proposed testing procedure leads to highly powerful calibration and auto-calibration tests. %This observation is confirmed by test different models to a simulated motor third party liability dataset 

\medskip
{\bf Organization.} The remainder of this manuscript is organized as follows. In the next section, we formally introduce the notions of strong calibration and auto-calibration, and present equivalent characterizations for both properties. Then, in Section \ref{test procedure}, we propose new testing procedures to assess these properties using boosting trees. Finally, in Section \ref{sec num example}, we show that the proposed testing procedures achieve a high power on a synthetic motor third party liability insurance dataset. The last section concludes this work. All mathematical proofs are provided in the appendix, which also includes an application of the proposed methodology to a real insurance dataset. %All mathematical proofs are provided in the appendix.

\section{Strong calibration and auto-calibration}

\label{sec cal auto cal}
\subsection{Mean estimation in regression modelling}

Denote by $Y$ the response variable and by $\bX$ the vector of features whose support is given by a feature space $\X$. Moreover, let $(\Omega, \mathcal{F}, \p)$ be the underlying probability space and assume that the response $Y$ is integrable. The mean estimation task in regression modelling consists in approximating the (unknown) {\it true conditional mean}
$$
\mu^{*} : \X \to \R, \quad \bx \mapsto \E[Y \, |\, \bX = \bx].
$$
For this, statistical models are used to fit a regression function $\widehat{\mu} : \X \to \R$ on some learning dataset. %that aims to be as close as possible to the true conditional mean $\mu^{*}:\X \to \R$. 
In practice, this fitting procedure is typically done by minimizing strictly consistent loss functions for mean estimation, which are functions whose expected values are uniquely minimized by the true conditional mean; we refer the reader to Gneiting \cite{Gneiting2} and Dimitriadis et al.~\cite{Dimitriadis}. The goal of this section is to introduce two different notions of calibration for the fitted regression function $\widehat{\mu} : \X \to \R$ and present equivalent formulations of these properties.

\subsection{The strong calibration property}

A regression function is strongly calibrated if it matches the true conditional mean for almost every (a.e.) set of features, i.e.,
$$
\widehat{\mu}(\bx) = \mu^*(\bx), \quad \textrm{for a.e. } \bx \in \X. 
$$
Using the response $Y$ and the vector of features $\bX$, we define this property as follows.
\begin{defi}
    \label{def calibration}
    A regression function $\widehat{\mu}: \X \to \R$ is strongly calibrated for $(Y, \bX)$ if 
    $$
    \widehat{\mu}(\bX) = \E[Y \, |\, \bX], \quad \p\textrm{-a.s.}
    $$
\end{defi}
It is important to note that the conditional expectation $\E[Y \, |\, \bX]$ is defined up to $\p\textrm{-}$nullsets, and that the measure-theoretic definition of conditional expectation, see Definition 23.4 in Jacod--Protter \cite{Jacod--Protter}, allows us to rewrite the above definition as follows. A strongly calibrated regression function $\widehat{\mu}: \X \to \R$ satisfies
\begin{equation}
    \label{cond exp calibration}
    \E[Y \mathds{1}_S] = \E[\widehat{\mu}(\bX) \mathds{1}_S],
\end{equation}
for all measurable sets $S \in \sigma(\bX)$. The following proposition can then be derived from this characterization.
\begin{prop}
    \label{prop equiv cal}
    Let $\widehat{\mu}: \X \to \R$ be a regression function and assume that $Y,\widehat{\mu}(\bX) \in L^1(\p)$. Then, $\widehat{\mu}: \X \to \R$ is strongly calibrated for $(Y, \bX)$ if and only if
    \begin{equation}
        \label{eq prop cal}
    \E\left[\left(Y- \widehat{\mu}(\bX)  \right) g(\bX) \right] = 0,
    \end{equation}
    for all measurable functions $g$ such that $\left(Y- \widehat{\mu}(\bX)  \right) g(\bX) \in L^1(\p)$. %such that $g(\bX) \in L^2(\p).$
\end{prop}
\begin{rem}
    \textnormal{If one further assumes that $Y, \widehat{\mu}(\bX) \in L^2(\p)$, one often refers to \eqref{eq prop cal} as the {\it orthogonal projection property} of conditional expectation. The reason is that, under this assumption, $\E[Y \, |\, \bX]$ corresponds to the orthogonal projection of \(Y\) onto the $L^2(\p)$-subspace of $\sigma(\bX)$-measurable functions. We refer the interested reader to Chapter 23 of Jacod--Protter \cite{Jacod--Protter}.}
\end{rem}

\subsection{The auto-calibration property}

Strong calibration is a strong property that requires the fitted regression function to coincide with the true conditional mean for a.e.~set of features $\bx \in \X$. A related notion is {\it auto-calibration}. It is defined as follows, see, e.g., Krüger--Ziegel \cite{Kruger}.

\begin{defi}
    \label{def auto calibration}
    A regression function $\widehat{\mu}:\X \to \R$ is auto-calibrated for $(Y, \bX)$ if
    \begin{equation*}
        \widehat{\mu}(\bX) = \E[Y \, | \, \widehat{\mu}(\bX)], \quad \p\textnormal{-a.s.}
    \end{equation*}
\end{defi}

The auto-calibration property means that by conditioning on a given mean estimate, the conditional expectation of the response matches this estimate. While this definition seems to be close to Definition \ref{def calibration}, the difference lies in the conditioning that now takes place with respect to the mean estimate $\widehat{\mu}(\bX)$ instead of the vector of features $\bX$, i.e., with respect to a coarser $\sigma$-algebra $\sigma(\widehat{\mu}(\bX)) \subseteq \sigma(\bX)$. Moreover, note that auto-calibration is a weaker property than strong calibration as for any strongly calibrated regression function $\widehat{\mu}:\X \to \R$, we have
\begin{equation*}
    %\label{cal implies auto-cal}
    \E[Y \, | \, \widehat{\mu}(\bX)] = \E[\E[Y \, | \, \bX] \, | \, \widehat{\mu}(\bX)] = \E[\mu^{*}(\bX) \, | \, \widehat{\mu}(\bX)] = \widehat{\mu}(\bX), \quad \p\textrm{-a.s.},
\end{equation*}
where we used in the first equality that $\sigma(\widehat{\mu}(\bX)) \subseteq \sigma(\bX)$ and in the last equality, that the regression function $\widehat{\mu}:\X \to \R$ is strongly calibrated. As above, Definition \ref{def auto calibration} can equivalently be expressed using the measure-theoretic definition of conditional expectation as follows. An auto-calibrated regression function $\widehat{\mu}:\X \to \R$ satisfies
\begin{equation}
\label{cond exp auto calibration}
    \E[Y \mathds{1}_S] = \E[\widehat{\mu}(\bX) \mathds{1}_S],
\end{equation}
for all measurable sets $S \in \sigma(\widehat{\mu}(\bX))$. The latter $\sigma$-algebra contains all measurable sets that can be uniquely defined by the values of the regression function. In particular, \eqref{cond exp auto calibration} implies that
\begin{equation*}
    %\label{reformulation auto-cal}
    \E[Y \, \mathds{1}_{\{a \leq \widehat{\mu}(\bX) \leq b\}}] = \E[\widehat{\mu}(\bX) \, \mathds{1}_{\{a \leq \widehat{\mu}(\bX) \leq b\}}],
\end{equation*}
for any values $a \leq b \in \R \cup \{\pm \infty\}$. Interestingly, by setting $a = - \infty$ and $b = \infty$, we retrieve the {\it global unbiasedness} property
$$
    \E[Y] = \E[\widehat{\mu}(\bX)],
$$
which shows that auto-calibration is a stronger property than global unbiasedness. %Moreover, by choosing $a = b = m$ for any value $m$ satisfying $\p(\widehat{\mu}(\bX) = m) > 0$, \eqref{reformulation auto-cal} implies
%$$
%m = \E[Y \, | \, \widehat{\mu}(\bX) = m].
%$$
%The measure theoretical definition of conditional expectation allows us to better interpret Definition \ref{def auto calibration} because \eqref{cond exp auto calibration} means that for an auto-calibrated regression function, the mean estimates are unbiased both globally and locally, where the latter term refers to subparts of the feature space that can solely be delimited by the values of the regression function. 

The characterization in \eqref{cond exp auto calibration} allows us to better interpret Definition \ref{def auto calibration}, because it means that the mean estimates provided by an auto-calibrated regression function are unbiased both globally and locally, where the latter term refers to subparts of the feature space that can solely be delimited by the values of the regression function. This property is of particular interest in applications where mean estimates within given specific groups have to be unbiased. This is for example the case in insurance pricing, where auto-calibration of a pricing system is a minimal requirement, as it ensures that each cohort of individuals paying a certain price is on average self-financing.
%; we refer to Wüthrich--Ziegel \cite{Wuethrich_Ziegel}. 
Similar to Proposition \ref{prop equiv cal}, the next result provides an equivalent formulation for auto-calibration.

\begin{prop}
    \label{prop equiv auto-cal}
    Let $\widehat{\mu}:\X\to\R$ be a regression function and assume that $Y,\widehat{\mu}(\bX)\in L^1(\p)$. Then, $\widehat{\mu}: \X \to \R$ is auto-calibrated for $(Y, \bX)$ if and only if
    \begin{equation}
        \label{eq prop auto-cal}
        \E\left[\left(Y- \widehat{\mu}(\bX)  \right) h(\widehat{\mu}(\bX)) \right] = 0,
    \end{equation}
    for all measurable functions $h$ such that $\left(Y- \widehat{\mu}(\bX)  \right) h(\widehat{\mu}(\bX)) \in L^1(\p).$
\end{prop}

We conclude this section by considering a toy example illustrating the difference between strong calibration and auto-calibration. This example shows that the set of auto-calibrated regression functions is substantially larger than the set of strongly calibrated ones. Moreover, it highlights that the granularity of auto-calibrated regression functions, given by $\sigma(\widehat{\mu}(\bX))$, can vary considerably.

\begin{example}
    \textnormal{Consider a feature space of only three elements $\X = \{A, B,C\}$ and let the true conditional mean of $Y\,|\bX$ be given by
    \begin{equation*}
        %\label{true example}
        \E[Y\, |\, \bX = \bx] =\mu^{*}(\bx) = \begin{cases}
        2, \quad \textnormal{for }\bx = A, \\
        6, \quad \textnormal{for }\bx = B, \\
        10, \quad \textnormal{for }\bx = C. \\
        \end{cases}
    \end{equation*}
    Moreover, assume that $\bX$ satisfies
    $$
    \p(\bX = A) = 1/2, \quad \p(\bX = B) = 1/4 \quad \textrm{and} \quad \p(\bX = C) = 1/4.
    $$
        The goal of this example is to evaluate the strong calibration and auto-calibration of the regression functions
    $$
        \widehat{\mu}_1(\bx) = \begin{cases}
        2, \quad \textnormal{for }\bx = A, \\
        8, \quad \textnormal{for }\bx = B, \\
        8, \quad \textnormal{for }\bx = C, \\
        \end{cases}
        \quad
        \textrm{and}
        \quad 
        \widehat{\mu}_2(\bx) = \begin{cases}
        5, \quad \textnormal{for }\bx = A, \\
        5, \quad \textnormal{for }\bx = B, \\
        5, \quad \textnormal{for }\bx = C. \\
        \end{cases}
        %\quad
        %\textrm{and}
        %\quad
        %\widehat{\mu}_3(\bx) = \begin{cases}
        %2, \quad \textnormal{for }\bx = A, \\
        %3, \quad \textnormal{for }\bx = B, \\
        %5, \quad \textnormal{for }\bx = C. \\
        %\end{cases}
    $$
    To this end, we first use \eqref{cond exp calibration} to assess strong calibration. That is, we check whether mean estimates are unbiased for all sets $S \in \sigma(\bX)$, which are given in this example by}
    \begin{multline*}
        \sigma(\bX) = \Big\{ \emptyset, \{\bX = A\}, \{\bX = B\}, \{\bX = C\}, \\ \{\bX \in A \cup B\}, \{\bX \in A \cup C\}, \{\bX \in B \cup C\}, \{\bX \in A \cup B \cup C\}  \Big\}.
    \end{multline*}
    \textnormal{For both regression functions, this property fails to hold for $S = \{\bX = B\}$ as we have}
    $$
    \E[Y \mathds{1}_S] = \E[\mu^*(\bX) \mathds{1}_S] = 1.5 \neq 2 = \E[\widehat{\mu}_1(\bX) \mathds{1}_S],
    $$
    \textnormal{and}
    $$
    \E[Y \mathds{1}_S] = \E[\mu^*(\bX) \mathds{1}_S] = 1.5 \neq 1.25 = \E[\widehat{\mu}_2(\bX) \mathds{1}_S].
    $$
    \textnormal{Thus, neither of the above regression functions is strongly calibrated and, in fact, as the random variable $\bX$ is discrete and its support is given by $\X$, the only strongly calibrated regression function in this example is given by $\mu^*: \X \to \R$. Next, we follow the same procedure to test for auto-calibration by using \eqref{cond exp auto calibration}. For this, note that}
    \begin{equation*}
        %\label{sigma 1}
        \sigma(\widehat{\mu}_1(\bX)) = \Big\{\emptyset, \{\bX = A\}, \{\bX \in B \cup C\}, \{\bX \in A \cup B \cup C\}\Big\},
    \end{equation*}
    \textnormal{and}
    \begin{equation*}
        %\label{sigma 2}
        \sigma(\widehat{\mu}_2(\bX)) = \Big\{\emptyset, \{\bX \in A \cup B \cup C\}\Big\}.
    \end{equation*}
    \textnormal{As these $\sigma$-algebras are strictly smaller than $\sigma(\bX)$, assessing the auto-calibration of both regression functions requires checking unbiasedness of mean estimates on fewer subparts of the covariate space and, in fact, it turns out that \eqref{cond exp auto calibration} holds for all sets in $\sigma(\widehat{\mu}_1(\bX))$ and $\sigma(\widehat{\mu}_2(\bX))$, respectively. This allows us to conclude that $\widehat{\mu}_1 : \X \to \R$ and $\widehat{\mu}_2 : \X \to \R$ are both auto-calibrated, but not strongly calibrated.
}
\EndExample
\end{example}

\section{Assessing strong calibration and auto-calibration using boosting trees}
\label{test procedure}
\subsection{Necessary conditions for strong calibration and auto-calibration}
\label{test procedure 1}
Assessing strong calibration is a challenging problem in practice as the true conditional mean is unknown. At first glance, the problem seems easier when considering auto-calibration because this property only involves the response $Y$ and the fitted regression function $\widehat{\mu} : \X \to \R$. For instance, the expected values in \eqref{cond exp auto calibration} could in principle be replaced by empirical means for a large amount of observations. However, a problem still remains as the equality in \eqref{cond exp auto calibration} has to be checked for all sets $S \in \sigma(\widehat{\mu}(\bX))$ whose number is typically uncountable and for which the number of available observations could be small. In practice, graphical methods are sometimes used to assess auto-calibration by plotting the empirical mean of binned responses against their corresponding mean estimates; see for example the {\it actual vs.~predicted} plot in Section 4.1.3 in Wüthrich et al.~\cite{Wüthrich2}. The problem of such approaches is that auto-calibration is only assessed for a few sets $S \in \sigma(\widehat{\mu}(\bX))$ and that the conclusion of this assessment largely depends on the chosen bins; we refer, for example, to Henzi et al.~\cite{Henzi}. This shows the need to have statistical tests to assess both strong calibration and auto-calibration. The goal of this section is to introduce new testing procedures to this end.%The goal of the next section is to introduce a new testing procedure based on boosting trees that allows us to evaluate the calibration and auto-calibration of fitted regression functions.

We start our discussion from Propositions \ref{prop equiv cal} and \ref{prop equiv auto-cal}, which provide equivalent conditions for strong calibration and auto-calibration, respectively. These conditions are difficult to evaluate in practice as the sets of functions $g$ in \eqref{eq prop cal} and $h$ in \eqref{eq prop auto-cal} are typically uncountable. Fissler et al.~\cite{Fissler} suggest to use a finite number $K$ of {\it test functions} $g_1, \dots,g_K$, respectively $h_1, \dots,h_K$, in order to simultaneously check the $K$ underlying necessary conditions using joint Wald tests. For example, in the case of strong calibration, the null-hypothesis of their test reads as
$$
\mathbb{H}_0 : \E\left[\left(Y- \widehat{\mu}(\bX)  \right) g_k(\bX) \right] = 0, \quad \textrm{for } 1 \leq k \leq K.  
$$
As each of the $K$ above equalities represents a necessary condition, their test allows them to reject strong calibration whenever the null-hypothesis $\mathbb{H}_0$ is rejected. The same procedure applies to auto-calibration as well. In practice, unfortunately, this approach has some limitations as pointed out by Fissler et al.~\cite{Fissler}. On the one hand, one needs a high number $K$ of test functions to be able to detect various kinds of violations of strong calibration and auto-calibration, respectively. On the other hand, the power of their test decreases with increasing $K$. Moreover, the choice of the test functions is also unclear; some examples are provided in Fissler et al.~\cite{Fissler}. 

In this section, we aim at introducing new procedures to test for strong calibration and auto-calibration by only considering a single necessary condition, respectively, while being able to detect various kinds of violations of the above properties. To this end, we assume that an i.i.d.~sample of data $\D = (Y_i, \bX_i)_{i\in \DD}$ following the law of a pair $(Y, \bX)$ is at our disposal, and that the regression function of interest is fitted on the realizations of some part of the data $\mathcal L = (Y_i, \bX_i)_{i\in \mathfrak{L}} \subset \D$. We refer to these realizations $\ell = (y_i, \bx_i)_{i\in \mathfrak{L}}$ as the {\it learning set} and denote the fitted regression function by $\widehat{\mu}_\ell : \X \to \R$, emphasizing its dependency on $\ell$. Under this framework, the strong calibration property reads as 
$$
\widehat{\mu}_\ell(\bX) =  \E[Y \, |\, \bX], \quad \p\textrm{-a.s.},
$$
and the auto-calibration property is given by
$$
\widehat{\mu}_\ell(\bX) = \E[Y \, |\, \widehat{\mu}_\ell(\bX)], \quad \p\textrm{-a.s.}
$$
For both of the above properties, note that the randomness only lies in the pair $(Y, \bX)$ because we only assess the strong calibration and auto-calibration of the fitted regression function $\widehat{\mu}_\ell~:~\X~\to~\R$ for the single set of realizations $\ell = (y_i, \bx_i)_{i\in \LL}$.

A standard approach for these assessments is to use the remaining data $\D \setminus \mathcal L$ for evaluating a test statistic; %that leads to a possible rejection of the null-hypothesis of calibration or auto-calibration;
we refer for example to Denuit et al.~\cite{Denuit2} and Wüthrich \cite{Wüthrich}. In our case, however, as we want to test for the necessary condition in \eqref{eq prop cal} and \eqref{eq prop auto-cal} by solely using a single test function, respectively, we take a slightly different approach. Indeed, we further divide the dataset $\D \setminus \mathcal{L}$ into a \textit{boosting set} $\B = (Y_i, \bX_i)_{i\in \mathfrak{B}}$ and a {\it test set} $\T = (Y_i, \bX_i)_{i\in \TT}$. The former set $\B$, along with the learning set $\l$, will be used to construct test functions $g_{\l, \B}$ and $h_{\l, \B}$ with the help of boosting trees. 
%The goal of these test functions will be to detect potential violations of the calibration and auto-calibration properties. 
The test set $\T$, for its part, will be used to evaluate a test statistic allowing us to possibly reject the null-hypothesis of strong calibration and auto-calibration, respectively. 
%the null-hypothesis
%\begin{equation}
    %\label{necessary cond Selim}
    %\mathbb{H}_0: \E[(Y-\widehat{\mu}_\l(\bX))  g_{\l, \B}(\bX)] = 0,
%\end{equation}
%in the case of calibration and
%\begin{equation}
    %\label{necessary cond Selim 2}
    %\mathbb{H}_0: \E(Y-\widehat{\mu}_\l(\bX)) h_{\l, \B}(\widehat{\mu}_\l(\bX))] = 0,
%\end{equation}
%in the case of auto-calibration. 
The use of each set is summarized in Figure \ref{Rectangles sets}. Moreover, the dependency of each considered function on the sets $\l, \B$ or $\T$ will be denoted by subscripts, and we use the notation $\X_\l, \X_\B$ and $\X_\T$ below for the sets of features appearing in the learning, boosting and test sets, respectively.

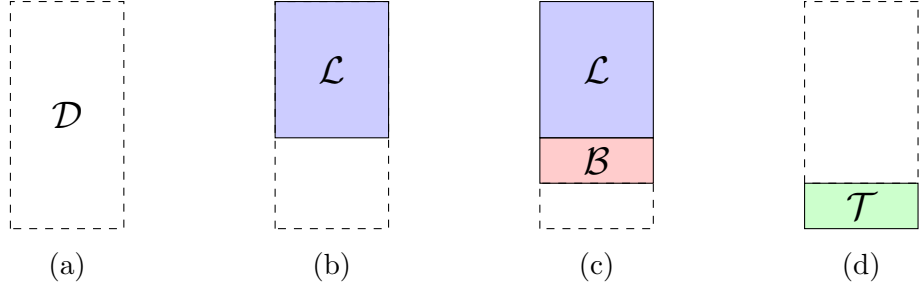
\begin{figure}[htb!]
\begin{center}
\begin{tikzpicture}[scale = 0.5]

% dimensions
\def\W{3}  % width of rectangles
\def\H{6}  % height

% --- Rectangle 0 ---
\draw[dashed] (0,0) rectangle (\W,\H);

\node at (0.5*\W,0.5*\H) {\Large $\mathcal{D}$};

% --- Rectangle 1 ---
\begin{scope}[xshift=7cm]
\fill[blue!20] (0,0.4*\H) rectangle (\W,\H);
\draw  (0,0.4*\H) rectangle (\W,\H);
\draw[dashed] (0,0) rectangle (\W,\H);

\node at (0.5*\W,0.7*\H) {\Large $\mathcal{L}$};

\node at (0.5*\W,-1) {(b)};

\end{scope}

% label
\node at (0.5*\W,-1) {(a)};

% --- Rectangle 2 ---
\begin{scope}[xshift=14cm]

% 60% area
\fill[blue!20] (0,0.4*\H) rectangle (\W,\H);

% 20% area
\fill[red!20] (0,0.4*\H) rectangle (\W,0.2*\H);

\node at (0.5*\W,0.7*\H) {\Large $\mathcal{L}$};
\node at (0.5*\W,0.3*\H) {\Large $\mathcal{B}$};

%\draw (0,0) rectangle (\W,\H);
\draw  (0,0.4*\H) rectangle (\W,\H);
\draw  (0,0.4*\H) rectangle (\W,0.2*\H);
\draw[dashed] (0,0) rectangle (\W,0.2*\H);

% label
\node at (0.5*\W,-1) {(c)};

\end{scope}

% --- Rectangle 3 ---
\begin{scope}[xshift=21cm]

% remaining 20%
\fill[green!20] (0,0.2*\H) rectangle (\W,0);

%\draw (0,0) rectangle (\W,\H);
\draw[dashed]  (0,0.2*\H) rectangle (\W,\H);
%\draw  (0,0.4*\H) rectangle (\W,0.2*\H);
\draw  (0,0.2*\H) rectangle (\W,0);
\node at (0.5*\W,0.1*\H) {\Large $\mathcal{T}$};

% label
\node at (0.5*\W,-1) {(d)};
\end{scope}

\end{tikzpicture}
\end{center}
\caption{(a) The large rectangle represents the full dataset $\D$. (b) The realizations of $\mathcal{L}$, denoted by the learning set $\l$, are used to fit the regression function $\widehat{\mu}_\l : \X \to \R$ of interest. (c) The same learning set $\l$ as well as the boosting set $\B$ are then used to construct test functions $g_{\l, \B}$ and $h_{\l, \B}$ using boosting trees. (d) The test set $\T$ is finally used to compute a test statistic leading to a possible rejection of the null-hypothesis of strong calibration and auto-calibration, respectively.}
\label{Rectangles sets}

\end{figure}
\subsection{Construction of the single test function using boosting trees}

\label{sec selecting test function}
Using a single test function in \eqref{eq prop cal} and \eqref{eq prop auto-cal} leads to the assessment of a necessary condition for strong calibration and auto-calibration, respectively. In general, this results in low statistical power as a regression function $\widehat{\mu}_\l : \X \to \R$ might, for example, satisfy \eqref{eq prop auto-cal} for some test functions $h$ while not being auto-calibrated. The choice of the test function thus plays a crucial role for the power of our tests and we consider the use of boosting trees in this regard.

Boosting is a class of ensemble methods that aims to improve an already fitted model by iteratively combining  weak learners with the aim of reducing the size of the residuals at each step. In order to assess strong calibration and auto-calibration, we propose to use regression trees as weak learners for fitting a competitive model to the regression model $\widehat{\mu}_\l : \X \to \R$ on both the learning and boosting sets. For this, we start from an initial regression function $\widehat{\mu}^\textrm{boost}_{(0)} : \X \to \R,$
e.g., the homogeneous mean on $\l$ and $\B$, and define successively
\begin{equation*}
    %\label{sum trees}
    \widehat{\mu}^\textrm{boost}_{(m)} : \X \to \R, \quad \bx \mapsto \widehat{\mu}^\textrm{boost}_{(m)} (\bx) = \widehat{\mu}^\textrm{boost}_{(m-1)}(\bx) + \beta_m t_m(\bx),
\end{equation*}
for $m \in \{1, \dots, M\}$. Above, the parameters $\beta_1, \dots, \beta_m \in \R$ correspond to pre-defined {\it learning rates} and the maps $t_m : \X \to \R$ belong to the class of regression trees with some pre-defined hyper-parameters that we denote by $\mathcal{M}_{tree}$. At each iteration $m$, these trees are selected by solving
\begin{equation*}
    %\label{minimzer loss tree}
  t_m \in \argmin\limits_{t \in \mathcal{M}_{tree}} \left(\sum_{i \in \LL \cup \BB} L\left(Y_i, \widehat{\mu}^\textrm{boost}_{(m-1)}(\bx) + \beta_m t_m(\bx_i)\right)\right),  
\end{equation*}
for a pre-selected strictly consistent loss function for mean estimation $L$. It is well known that for large values of iteration steps $M$, the resulting regression function $\widehat{\mu}^\textrm{boost}_{(M)} : \X \to \R$ tends to overfit, i.e., it starts modelling the noisy part of the learning set $\l$ and the boosting set $\B$ instead of its systematic part. A possible solution to mitigate this issue is to carefully select the learning rates $\left(\beta_m\right)_m$ and to apply early stopping; we refer to Friedman \cite{Friedman} and to Chapter 10 of Hastie et al.~\cite{Hastie}. We do not further elaborate on these aspects here. A practical implementation of the boosting algorithm that addresses overfitting will be provided in Section \ref{sec num example}, below.

Throughout this paper, we call $\widehat{\mu}^\textrm{boost}_{(M)} : \X \to \R$ the {\it benchmark gradient boosting model} (GBM) and denote it by $\widehat{\mu}^\textrm{boost}_{\l,\B} : \X \to \R$, because it is fitted on both the learning and boosting sets. Its goal is to primarily capture the systematic part of the sets $\ell$ and $\B$, i.e., we assume that appropriate measures have been taken during the fitting stage to mitigate overfitting. Under this assumption, boosting trees are known to achieve strong out-of-sample predictive performance; see, e.g., Friedman \cite{Friedman}. As the fitting procedure of this new model makes use of an additional set, the boosting set $\mathcal{B}$, we expect the benchmark GBM $\widehat{\mu}^\textrm{boost}_{\l,\B} : \X \to \R$ to achieve a higher predictive power than the originally fitted regression function $\widehat{\mu}_{\ell} : \mathcal{X} \to \mathbb{R}$. 

Our proposal is to detect violations of strong calibration and auto-calibration for the originally fitted regression function by comparing its mean estimates with those produced by the benchmark GBM. For instance, on the one hand, the inequality  
\begin{equation}
    \label{case 1}
    \widehat{\mu}^\textrm{boost}_{\l, \B}(\bx) > \widehat{\mu}_\l(\bx)
\end{equation}
indicates that the initial regression function $\widehat{\mu}_\l : \X \to \R$ might exhibit a systematic negative bias for the sets of features $\bx \in \X$ satisfying \eqref{case 1}, whereas, on the other hand, the inequality
\begin{equation}
    \label{case 2}
    \widehat{\mu}^\textrm{boost}_{\l, \B}(\bx) < \widehat{\mu}_\l(\bx)
\end{equation}
hints to a possible systematic positive bias. Using the originally fitted regression function $\widehat{\mu}_\l~:~\X\to\R$ as well as the benchmark GBM $\widehat{\mu}^\textrm{boost}_{\l, \B}~:~\X\to\R$, test functions $g_{\l, \B}$ and $h_{\l, \B}$ can then be constructed with the aim of pushing the expected values in \eqref{eq prop cal} and \eqref{eq prop auto-cal} away from $0$ whenever the null-hypotheses of strong calibration and auto-calibration are violated. This can be done, for example, by multiplying the residuals $Y - \widehat{\mu}_\l(\bX)$ by a positive value when \eqref{case 1} holds, and by a negative value when \eqref{case 2} holds. We make this choice below, where we provide general procedures to assess strong calibration and auto-calibration.

\subsection{Assessing strong calibration}
\label{sec assessing calibration}

In this section, our aim is to assess strong calibration by testing \eqref{eq prop cal} for a single test function $g_{\l, \B}~:~\X\to\R$ that is constructed using the learning set $\l$ and the boosting set $\B$. For this, we rely on the following proposition, whose proof is provided in the appendix.

\begin{prop}
    \label{prop theory cal}
    Let $\T =(Y_i, \bX_i)_{i \in \TT}$ be a sequence of i.i.d.~random vectors following the law of some independent pair $(Y, \bX)$. Moreover, let $\widehat{\mu}_\ell : \X \to \R$ be a measurable function and $g_{\ell, \B}: \X \to \R$ be a $\sigma(\B)$-measurable function for some set of i.i.d~random vectors $\B =(Y_i, \bX_i)_{i \in \BB}$ that is independent of $\T$. Then, the random variables
    \begin{equation}
        \label{prop theory cal Z_i}
        Z_i = (Y_i - \widehat{\mu}_\ell(\bX_i)) g_{\ell, \B}(\bX_i), \quad i \in \TT,
    \end{equation}
    are conditionally i.i.d.~given $\B$. If these random variables further satisfy $$0 < \textnormal{Var}(Z_i \, | \, \B) < \infty, \quad \p\textrm{-a.s.,}$$ we have, under the null-hypothesis of strong calibration of $\widehat{\mu}_\ell : \X \to \R$, that
    \begin{equation}
        \label{convergence cal}
        T_n^{{\textnormal{sc}}} = \frac{\bar{Z}}{\sqrt{S_Z^2/n}} \stackrel{d}{\longrightarrow} \mathcal{N}(0,1),
    \end{equation}
    as $n = |\T| \to \infty$, with
    \begin{equation} 
    \label{Z bar S squared}
    \bar{Z} = \frac{1}{n} \sum_{i=1}^n Z_i \quad \textrm{and} \quad S_Z^2 = \frac{1}{n-1} \sum_{i=1}^n (Z_i - \bar{Z})^2.
    \end{equation}
\end{prop} 

\begin{rem}
    \label{rem prop calibration}
    \textnormal{The randomness in \eqref{convergence cal} lies simultaneously in the boosting set $\B$ and the test set $\T$. Moreover, the above result holds for any fixed size of the boosting set, as long as the size of the test set $|\T|$ goes to infinity. However, although the size $|\B|$ does not affect the convergence of the test statistic $T_n^{{\textrm{sc}}}$ under the null-hypothesis of strong calibration of $\widehat{\mu}_\ell : \X \to \R$, we emphasize that it plays a role when this null-hypothesis is violated. Indeed, the larger the boosting set $\B$, the better the benchmark GBM should be able to detect violations of strong calibration and lead to the rejection of the null-hypothesis in the latter case.}
\end{rem}

Following the discussion in Section \ref{sec selecting test function}, Proposition \ref{prop theory cal} naturally motivates the general procedure for constructing strong calibration tests, below. The idea of this procedure is to introduce a test statistic based on weighted residuals, where the weights are given by a test function $g_{\l,\B}~:~\X~\to~\R$, see \eqref{prop theory cal Z_i} and \eqref{convergence cal}. This test function will be constructed using the benchmark GBM $\widehat{\mu}^\textrm{boost}_{\l, \B} : \X \to \R$ in such a way that it takes positive values whenever positive biases are detected in the originally fitted regression function $\widehat{\mu}_{\l} : \X \to \R$, and negative values whenever negative biases are detected, see Section \ref{sec selecting test function}. The following procedure provides a general framework; an explicit implementation of the resulting test is presented in Section \ref{sec num example test cal}.

\noindent\hrulefill\\%[-0.35em]
\centerline{\textsc{Procedure to construct a strong calibration test using boosting trees}}
\vspace{-0.35em}
\noindent\hrulefill

\bigskip

\begin{enumerate}
  %\item Fit a competitive model to $\widehat{\mu}_\l:\X \to \R$ using boosting trees on the learning set $\l$ and the boosting set $\B$. This results in a new regression function $\widehat{\mu}^\textrm{boost}_{\l, \B} : \X \to \R$.
  \item Fit a benchmark GBM $\widehat{\mu}^\textrm{boost}_{\l, \B} : \X \to \R$ on the learning set $\l$ and the boosting set $\B$.
  \item Select a test function $g_{\l, \B} : \X \to \R$ such that for $\bx \in \X$,
    \begin{itemize}
        \item $g_{\l, \B}(\bx) > 0$, whenever $\widehat{\mu}^\textrm{boost}_{\l, \B}(\bx) > \widehat{\mu}_\l(\bx)$,%, i.e., when we believe that $\widehat{\mu}_\l : \X \to \R$ is under-fitting.
        \item $g_{\l, \B}(\bx) = c \in \R\setminus\{0\}$, whenever $\widehat{\mu}^\textrm{boost}_{\l, \B}(\bx) = \widehat{\mu}_\l(\bx)$,%, i.e., when we believe that $\mu : \X \to \R$ is neither over-fitting, nor under-fitting.
        \item $g_{\l, \B}(\bx) < 0$, whenever $\widehat{\mu}^\textrm{boost}_{\l, \B}(\bx) < \widehat{\mu}_\l(\bx)$.%, i.e., when we believe that $\mu : \X \to \R$ is over-fitting.
    \end{itemize}
    \item Compute the test statistic $T_n^{\textrm{sc}}$ in \eqref{convergence cal} on the test set $\T$, i.e., for the random variables
    $$
    Z_i = (Y_i - \widehat{\mu}_{\l}(\bX_i)) g_{\l, \B}(\bX_i), \quad 1 \leq i \leq n,
    $$
    where $n = |\T|$.
    \item Reject the strong calibration of the regression function $\widehat{\mu}_\l : \X \to \R$ at the pre-specified confidence level $1-\alpha \in (0,1)$ whenever
    $$
        |T_n^{\textrm{sc}}| > \Phi^{-1}(1-\alpha/2),
    $$
    where $\Phi(\cdot)$ denotes the standard Gaussian distribution.
\end{enumerate}
\noindent\hrulefill

\begin{rem}
    \label{rem theory cal}
    \textnormal{The conditions on the test function $g_{\l, \B} : \X \to \R$ in the second step aim at making the test statistic $T_n^{\textrm{sc}}$ positive whenever the strong calibration assumption is violated, see Section \ref{sec selecting test function}. Such a statement assumes that the boosted regression function $\widehat{\mu}^\textrm{boost}_{\l, \B}: \X \to \R$ is able to successfully detect those violations, and we emphasize that even if this is not the case, the proposed testing procedure remains valid. It might, however, exhibit a low power. Moreover, note that, above, we impose that the test function differs from $0$ whenever $\widehat{\mu}^\textrm{boost}_{\l, \B}(\bx) = \widehat{\mu}_\l(\bx)$. The reason for this requirement is to avoid the degenerate case $T_n^{\textrm{sc}} = 0, \, \p\textrm{-a.s.}$, which occurs when $\widehat{\mu}_\ell(\bX) = \widehat{\mu}^{\textrm{boost}}_{\ell,\mathcal{B}} (\bX), \, \p\textrm{-a.s.}$ For most pairs $(Y,\bX)$, the latter event has probability zero and in the other cases where this could happen, the corresponding test function would satisfy $g_{\ell,\mathcal{B}}(\bx) \equiv c$ for a.e.~$\bx\in\mathcal{X}$, so that we only assess global unbiasedness. Interestingly, in this situation, the value of the test statistic $T_n^{\textrm{sc}}$ would not depend on the choice of the constant $c$.} %Finally, we point out again that as the null-hypothesis \eqref{necessary calibration} is a necessary condition for calibration, see Proposition \ref{prop equiv cal}, our test allows us to test for calibration.}
\end{rem} 

\subsection{Assessing auto-calibration}

\label{sec assessing auto calibration}

A similar result to Proposition \ref{prop theory cal} enables us to assess auto-calibration. The only difference is that, this time, the test function is not defined on the feature space $\X$, but on the {\it mean estimate space}
$$
{\widehat{\mu}}_\l (\X) = \left\{  {\widehat{\mu}}_\l(\bx) : \bx \in \X  \right\}.
$$
Thus, the test statistic for auto-calibration 
\begin{equation}
        \label{convergence auto cal}
        T_n^{\textrm{ac}} = \frac{\bar{Z}}{\sqrt{S_Z^2/n}},
    \end{equation}
is obtained by replacing the random variables $Z_i$ in \eqref{prop theory cal Z_i} with
\begin{equation}
    \label{prop theory auto-cal Z_i}
        Z_i = (Y_i - \widehat{\mu}_\ell(\bX_i)) h_{\ell, \B}(\widehat{\mu}_\ell(\bX_i)), \quad i \in \TT,
\end{equation}
where $\bar{Z}$ and $S_Z^2$ are given in \eqref{Z bar S squared}. Under the assumptions of Proposition \ref{prop theory cal} and under the null-hypothesis of auto-calibration of $\widehat{\mu}_\l : \X \to \R$, the test statistic $T_n^{\textrm{ac}}$ satisfies the same convergence result as in \eqref{convergence cal}, which leads to the general procedure for testing auto-calibration, below.

The main difference with respect to the procedure in Section \ref{sec assessing calibration} will lie in the second step. Indeed, as the test function $h_{\l, \B} : \widehat{\mu}_\l(\X) \to \R$ is defined on the mean estimate space, we aim now at comparing the values of the initially fitted regression function $\widehat{\mu}_\l: \X \to \R$ with the values of the map
\begin{equation*}
    %\label{mu tilde}
    \widehat{\mu}^\textrm{boost}_{\l, \B} |_{\widehat{\mu}_\l} : {\widehat{\mu}}_\l (\X) \to \R, \quad m \mapsto \E\left[\left.\widehat{\mu}^\textrm{boost}_{\l, \B}(\bX) \, \right| \, \widehat{\mu}_{\l}(\bX) = m\right].
\end{equation*}
The problem is that this map is typically unavailable in practice,
because the distribution of the feature vector $\bX$ is unknown.
We thus approximate it by constructing a function $\widetilde{\mu}^\textrm{boost}_{\l, \B}~:~{\widehat{\mu}}_\l (\X)~\to~\R$ on the mean estimate space. To this end, we first define
\begin{equation}
    \label{impossible map}    
    \widetilde{\mu}^\textrm{boost}_{\l, \B} (m) = \frac{\sum_{i \in \LL \cup \BB} \widehat{\mu}^\textrm{boost}_{\l, \B}(\bx_i) \mathds{1}_{\{\widehat{\mu}_\ell(\bx_i) = m\}}}{\sum_{i \in \LL \cup \BB} \mathds{1}_{\{\widehat{\mu}_\ell(\bx_i) = m\}}},
\end{equation}
for $m \in {\widehat{\mu}}_\l (\X_\l \cup \X_\B)$. Then, we extend this definition to the
entire mean estimate space by linearly interpolating
between the above values over the
interval
\[
    \left[
        \min_{i\in\LL\cup\BB}\widehat{\mu}_\ell(\bx_i),
        \max_{i\in\LL\cup\BB}\widehat{\mu}_\ell(\bx_i)
    \right],
\]
and keeping the end points' values outside this interval. As for strong calibration, we provide a general procedure for assessing auto-calibration; an explicit implementation of the resulting test is presented in Section \ref{sec num example test auto cal}.

\noindent\hrulefill\\%[-0.35em]
\centerline{\textsc{Procedure to construct an auto-calibration test using boosting trees}}
\vspace{-0.35em}
\noindent\hrulefill

\bigskip

\begin{enumerate}
  \item Fit a benchmark GBM $\widehat{\mu}^\textrm{boost}_{\l, \B} : \X \to \R$ on the learning set $\l$ and the boosting set $\B$, and construct a map $\widetilde{\mu}^\textrm{boost}_{\l, \B} : \widehat{\mu}_\l (\X) \to \R$ as above.
  \item Select a function $h_{\l, \B} : \widehat{\mu}_\ell(\X) \to \R$ such that for $m \in \widehat{\mu}_\ell(\X)$,
    \begin{itemize}
        \item $h_{\l, \B}(m) > 0$, whenever $\widetilde{\mu}^\textrm{boost}_{\l, \B}(m) > m$,%, i.e., when we believe that $\mu : \X \to \R$ is under-fitting.
        \item $h_{\l, \B}(m) = c \in \R \setminus \{0\}$, whenever $\widetilde{\mu}^\textrm{boost}_{\l, \B}(m) = m$,%, i.e., when we believe that $\mu : \X \to \R$ is neither over-fitting, nor under-fitting.
        \item $h_{\l, \B}(m) < 0$, whenever $\widetilde{\mu}^\textrm{boost}_{\l, \B}(m) < m$.%, i.e., when we believe that $\mu : \X \to \R$ is over-fitting.
    \end{itemize}
    \item Compute the test statistic $T_n^{\textrm{ac}}$ in \eqref{convergence auto cal} on the test set $\T$, i.e., for 
    \begin{equation*}
        %\label{Z_i auto-cal}
        Z_i = (Y_i - \widehat{\mu}_{\l}(\bX_i)) h_{\l, \B}(\widehat{\mu}_{\l}(\bX_i)), \quad 1 \leq i \leq n,
    \end{equation*}
    where $n = |\T|$.
    \item Reject the auto-calibration of the regression function $\widehat{\mu}_\l : \X \to \R$ at the pre-specified confidence level $1-\alpha \in (0,1)$ whenever
    $$
        \left|T_n^{\textrm{ac}}\right| > \Phi^{-1}(1-\alpha/2),
    $$
    where $\Phi(\cdot)$ denotes the standard Gaussian distribution.
\end{enumerate}
\noindent\hrulefill

\begin{rem}
    \textnormal{Other methods could have been used for extending the domain of the map
$\widetilde{\mu}^{\textrm{boost}}_{\ell,\B}$ to the entire mean estimate
space as, for example, cubic interpolation or spline-based methods. In fact, this interpolation step could even be avoided by
comparing the initially fitted regression function
$\widehat{\mu}_\ell:\X\to\R$ with a benchmark GBM fitted directly on the
one-dimensional mean estimate space $\widehat{\mu}_\ell(\X)$. In that case, the benchmark GBM would not target the map in \eqref{impossible map}, but would rather aim to approximate the recalibration map
\begin{equation*}
    \widehat{\mu}_{\l, \B}^{\textrm{rec}} |_{\widehat{\mu}_\l} : {\widehat{\mu}}_\l (\X) \to \R, \quad m \mapsto \E\left[\left.Y\, \right| \, \widehat{\mu}_{\l}(\bX) = m\right],
\end{equation*}
using the learning and boosting sets. While being simpler to implement, such an approach would discard the additional information
contained in the feature vector $\bX$ as it would only use the information provided by the (possibly much) coarser $\sigma$-algebra $\sigma(\widehat{\mu}_\ell(\bX))
    \subseteq \sigma(\bX).$ Therefore, we do not pursue this approach here. However, we emphasize that whether the direct
one-dimensional approach performs better or worse in finite samples depends
on the underlying responses, feature vectors and considered regression function. Thus, we regard it as an interesting alternative.}
\end{rem}

We conclude this section by pointing out that the auto-calibration test of Denuit et al.~\cite{Denuit2} is also based on the measure-theoretic definition of conditional expectation as it seeks to verify whether
$$
    \E[Y \mathds{1}_{\{\widehat{\mu}_\l(\bX) \leq t\}}] = \E[\widehat{\mu}(\bX) \mathds{1}_{\{\widehat{\mu}_l(\bX) \leq t\}}],
$$
holds for all $t \in \R$, which is an equivalent condition to \eqref{cond exp auto calibration}. %However, in contrast to our proposed testing procedure, their test allows us to test for an equivalent condition of auto-calibration.
To this end, these authors consider the curve
\begin{equation}
    \label{accumaltion Denuit}
    \alpha \in (0,1) \longmapsto A_n(\alpha) = \frac{1}{\sqrt{|\T|}}\sum_{i \in \TT} (Y_i-\widehat{\mu}_\ell(\bX_i)) \mathds{1}_{\left\{\widehat{\mu}_\ell(\bX_i) \leq \widehat{F}^{-1}_{\widehat{\mu}_\ell}(\alpha)\right\}},
\end{equation}
where $\widehat{F}^{-1}_{\widehat{\mu}_\ell}(\cdot)$ denotes the empirical quantile of the distribution of $\widehat{\mu}_\ell(\bX)$. This curve is then used to compute the test statistic
\begin{equation*}
    T_n = \sup_{\alpha \in (0,1)} |A_n(\alpha)|,
\end{equation*}
whose distribution is unknown. Therefore, they propose to compute the critical level of rejection for their test statistic using non-parametric Monte Carlo simulation approximations.

In principle, their test could be modified such that the accumulation of responses and mean estimates in \eqref{accumaltion Denuit} takes place in the opposite direction as  
$$
    \E[Y \mathds{1}_{\{\widehat{\mu}_\ell(\bX) \geq t\}}] = \E[\widehat{\mu}_\ell(\bX) \mathds{1}_{\{\widehat{\mu}_\ell(\bX) \geq t\}}],
$$
for all $t \in \R$, is equivalent to \eqref{cond exp auto calibration} too. In fact, any other arbitrary (and pre-defined) order could be chosen as well. 

We show in the next section that our auto-calibration test achieves higher empirical power than the test of Denuit et al.~\cite{Denuit2} for a large insurance dataset. This may seem surprising at first glance because their test assesses an equivalent condition for auto-calibration, whereas we only test for a necessary condition. Consequently, our test cannot be expected to be uniformly more powerful, but it may achieve higher power when the benchmark GBM successfully identifies the regions where violations of auto-calibration are most likely to occur. In fact, this is the main difference with the test of Denuit et al.~\cite{Denuit2}, which is based on a Kolmogorov--Smirnov type test statistic that considers the entire support of the mean estimates $\widehat{\mu}_\ell(\bX)$ simultaneously instead of focusing on parts of the support where violations are most likely to happen. Interestingly, we point out that our idea of first learning violations of auto-calibration before defining the test statistic could similarly be used in their setup, leading, for example, to compute the above test statistic for only a subpart of the mean estimate space.

\section{Numerical Example}
\label{sec num example}
\subsection{Dataset}
\label{sec dataset}
We apply in this section the two testing procedures presented in Section \ref{test procedure} to a Swiss motor third party liability insurance (MTPL) dataset introduced by Wüthrich--Buser \cite{Wüthrich--Buser}. This dataset has been synthetically constructed based on the French MTPL real dataset available in the
{\sf R} \cite{R Core Team} package {\tt CASdatasets} hosted by Dutang--Charpentier \cite{Dutang--Charpentier}. It contains information on insurance policies and claim frequencies of $n =$ 500,000 Swiss car drivers\footnote{The dataset can be downloaded under \url{https://people.math.ethz.ch/~wueth/Lecture/MTPL_data.csv}.}. For each policy $1 \leq i \leq n$, the number of claims $N_i \in \mathbb{N}$ that occurred during an exposure period $v_i \in (0,1]$ (years-at-risk) is available. The sum of the exposures $(v_i)_{i=1}^n$ is equal to 253,022 years, indicating that some policyholders were
covered for a period of less than one year in this portfolio, and as one might expect in motor liability insurance, most policies do not suffer any claim; see Table \ref{Tab:description data}.

\begin{table}[htb!]
    \centering
    \begin{tabular}{l c c c c}
      \toprule
      Number of claims per policy& 0 & 1 & 2 & 3\\
      \midrule
      Number of policies & 475,153 & 23,773 & $1012$ & $62$ \\
      Total exposure & 235,142 & 17,021 & $811$ & $48$ \\
      \bottomrule
    \end{tabular}
    \caption{Number of policies and total exposure within the portfolio that is split with respect to the number of claims per policy.}
    \label{Tab:description data}
\end{table} 
In addition to the number of claims and the exposure, features containing information on each policy are provided and collected into vectors $\bX_i = (X_{i,1}, \dots, X_{i,8})^\top$ as follows:
\begin{itemize}
    \item $X_{i,1}$ : age of the driver ({\tt age}), continuous feature in $\{18, \dots, 90\}$ years;
    \item $X_{i,2}$ : age of the car ({\tt ac}), continuous feature in $\{0, \dots, 35\}$ years;
    \item $X_{i,3}$ : power of the car ({\tt power}), continuous feature in $\{1, \dots, 12\}$;
    \item $X_{i,4}$ : fuel type of the car ({\tt gas}), binary feature (regular petrol/diesel);
    \item $X_{i,5}$ : anonymized brand of the car ({\tt brand}), categorical feature with $11$ labels;
    \item $X_{i,6}$ : area code ({\tt area}), categorical feature with $6$ labels;
    \item $X_{i,7}$ : density at the living place of the driver ({\tt dens}), continuous feature in $[1, 27000]$;
    \item $X_{i,8}$ : Swiss canton of the car license plate ({\tt ct}), categorical feature with $26$ labels.
\end{itemize}

Using a synthetically generated dataset has the advantage that the true annual frequency of claims is available for each policyholder. This frequency ranges from $0.004$ to $1.396$ for the considered portfolio and is low for most policies; see Figure \ref{Fig : Density Freq}. We refer to Appendix A in Wüthrich--Buser \cite{Wüthrich--Buser} for an extended description of the dataset.

\begin{figure}[htb!]
\begin{center}
\includegraphics[width=0.8\textwidth]{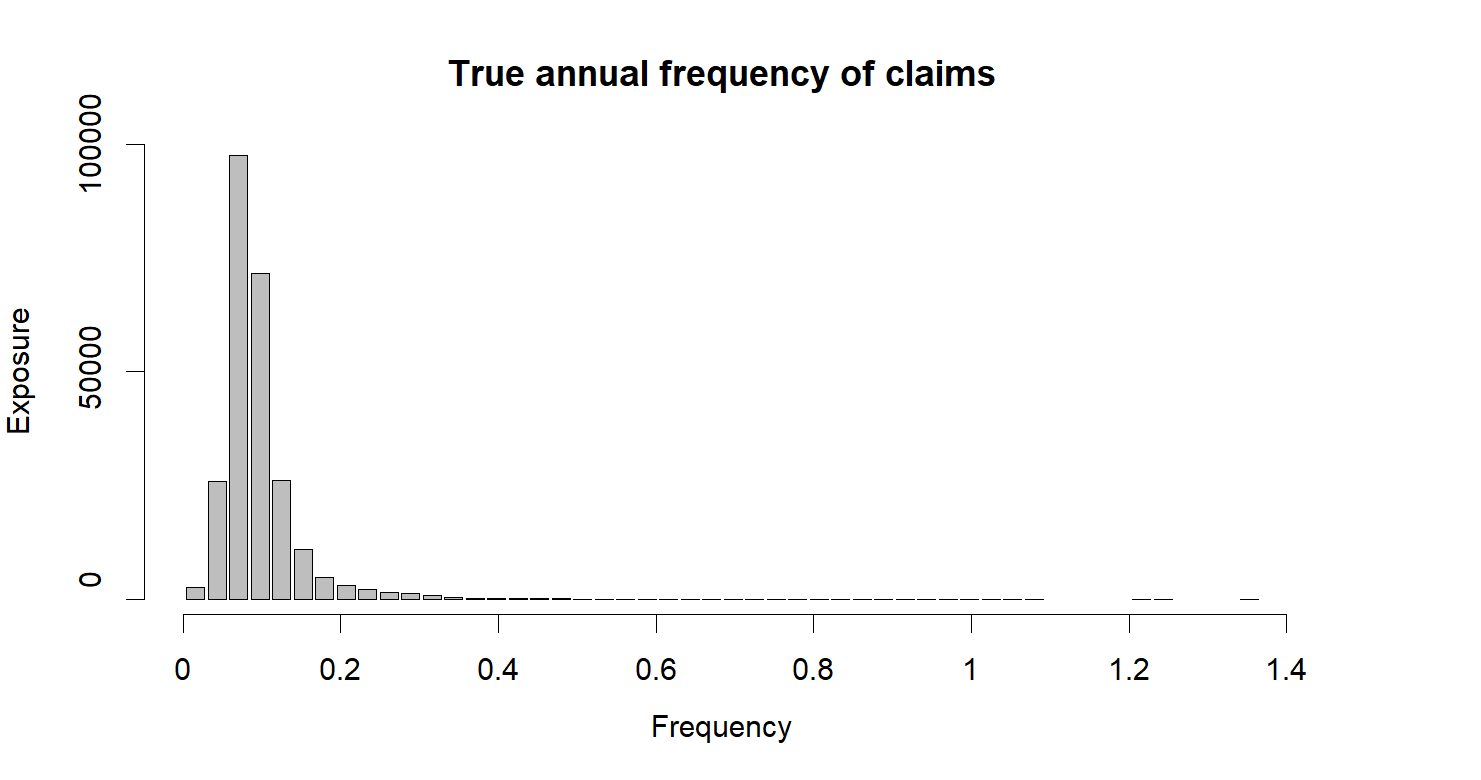}
\end{center}
\vspace{-.7cm}
\caption{Histogram of the true annual frequency of claims for each policy. The $y$-axis provides the aggregated exposures.}
\label{Fig : Density Freq}
\end{figure}

\subsection{Considered models}

\label{sec considered models}
The dataset under consideration $\D = \{(Y_i, \bX_i, v_i)\}_{i=1}^n$ consists of i.i.d.~triplets following the law of some triplet $(Y, \bX, v)$. Moreover, we know that the number of claims for each policyholder was generated from 
$$
N_i \sim \textrm{Poi} (\mu^{*}(\bX_i)v_i), \quad 1 \leq i \leq n,
$$
where $\mu^{*} : \X \to (0, \infty)$ denotes the true annual frequency of claims. This frequency actually provides the true conditional mean of the responses $Y_i = N_i/v_i$, i.e.,
$$
    \E[Y_i \, | \, \bX_i] =  \mu^{*}(\bX_i),\quad 1 \leq i \leq n.
$$
We follow Section \ref{test procedure 1} and divide the available dataset into a learning, a boosting and a test set. Then, we fit seven different regression functions $\widehat{\mu}_\l : \X \to (0, \infty)$ on the learning set $\ell$ by minimizing the {\it Poisson deviance loss}
$$
    L(\l, \widehat{\mu}_\l) = \frac{2}{n} \sum_{i \in \LL} v_i\left(\widehat{\mu}_\l(\bx_i) - Y_i - Y_i \log\left(  \frac{\widehat{\mu}_\l(\bx_i)}{Y_i}\right)\right),
$$
with the usual convention that $Y_i\log(\widehat{\mu}_\ell(\bx_i)/Y_i)=0,$ whenever $Y_i=0$. The choice of this loss function is motivated by two considerations. First, it is a strictly consistent loss function for mean
estimation. Second, minimizing the Poisson deviance loss is equivalent to
maximizing the Poisson likelihood; see Chapter 5.1.7 of Wüthrich--Merz
\cite{WM2023}. In addition to the seven fitted models, we consider
the true conditional mean below, leading to a total of eight regression functions:

\begin{itemize}
    \item $\mu^{*} : \X \to (0, \infty)$, the true conditional mean.
    \item $\widehat{\mu}_{\l}^{\textrm{hom}} : \X \to (0, \infty)$, the null model corresponding to the homogeneous (weighted) mean of the learning set, i.e.,
    \begin{equation}
        \label{hom mean}
        \widehat{\mu}_{\l}^{\textrm{hom}}(\bx) = \frac{\sum_{i \in \LL} v_i Y_i}{\sum_{i \in \LL} v_i}, \quad \textrm{for } \bx \in \X.
    \end{equation}
    \item $\widehat{\mu}_{\l}^{\textrm{GLM}} : \X \to (0, \infty)$, a generalized linear model (GLM) using all covariates. It corresponds to the model GLM3 in Wüthrich--Buser \cite{Wüthrich--Buser}, i.e., we pre-process the features $\tt{age}$, $\tt{ac}$ and $\tt{dens}$ in the same way.
    \item $\widehat{\mu}_{\l}^{\textrm{GAM}} : \X \to (0, \infty)$, a generalized additive model (GAM) using all covariates. It corresponds to the model GAM2 in Denuit et al.~\cite{Denuit2}.
    \item $\widehat{\mu}_{\l}^{\textrm{DNN1}} : \X \to (0, \infty)$, a deep neural network (DNN) using all covariates. It mainly corresponds to the model DNN2 in Wüthrich--Buser \cite{Wüthrich--Buser}, with the exception that continuous variables are standardized in the pre-processing stage (see formula (5.15) in Wüthrich--Buser \cite{Wüthrich--Buser}), the batch size is taken to be $100$ and the maximal number of epochs is $100$. During training, we first split the learning set into an $80\%$ training subset and a $20\%$ validation subset, and select the optimal number of epochs by minimizing the out-of-sample Poisson deviance loss on the validation subset. Next, the model is refitted on the entire learning set using the selected number of epochs. This procedure aims to mitigate overfitting; we refer the reader to Chapter 5 in Wüthrich--Buser \cite{Wüthrich--Buser} for an extended description of such a method.
    \item $\widehat{\mu}_{\l}^{\textrm{DNN2}} : \X \to (0, \infty)$, a deep neural network using all covariates. The difference with the previous model is that, this time, the model is directly fitted on the entire learning set using $100$ epochs, leading to overfitting.
    \item $\widehat{\mu}_{\l}^{\textrm{GBM1}} : \X \to (0, \infty)$, a GBM using all covariates. This model is fitted as in Listing 7.2 in Wüthrich--Buser \cite{Wüthrich--Buser}, with the difference that the depth of the trees is chosen to be $3$, the maximum number of iterations is $500$, the learning rate is $0.5$, the shrinkage is $0$ and the minimal observations per bucket is $5000$. In the fitting stage, we use the same approach as for the model DNN1 to select the optimal number of boosting steps in order to prevent overfitting. Then, we fit the model on the entire learning set using the selected number of steps.
    \item $\widehat{\mu}_{\l}^{\textrm{GBM2}} : \X \to (0, \infty)$, a GBM using all covariates. The difference with the previous model is that the number of iterations is now chosen to be equal to the large value $500$, leading to overfitting.
\end{itemize}

%All the above regression functions, except the true conditional mean, have been fitted on the learning set $\l$. 
The goal of the next sections is to assess the strong calibration and auto-calibration of these regression functions by considering two different cases:
\begin{itemize}
    \item {\it Case 1}. The sets $\mathcal{L}, \B$ and $\T$ represent $60 \%, 20 \%$ and $20 \%$ of the dataset $\D$, respectively.
    \item {\it Case 2}. The sets $\mathcal{L}, \B$ and $\T$ represent $80 \%, 10 \%$ and $10 \%$ of the dataset $\D$, respectively.
\end{itemize}

As this dataset contains an exposure $v$, both calibration properties have to be understood as follows. A regression model $\widehat{\mu}_\l : \X \to \R$ is strongly calibrated for $(Y, \bX, v)$ if
\begin{equation}
    \label{calibration v}
    \E[v\widehat{\mu}_\l(\bX) \, |\, \bX] = \E[vY \, |\, \bX], \quad \p\textrm{-a.s.},
\end{equation}
and it is auto-calibrated for $(Y, \bX, v)$ if 
\begin{equation}
    \label{auto-calibration v}
\E[v\widehat{\mu}_\l(\bX) \, |\, \widehat{\mu}_\l(\bX)] = \E[vY \, |\, \widehat{\mu}_\l(\bX)], \quad \p\textrm{-a.s.}
\end{equation}
Therefore, the residuals involved in the definition of the random variables $Z_i$ in \eqref{prop theory cal Z_i} and \eqref{prop theory auto-cal Z_i} will be taken on the level of the number of claims $N = vY$ when computing the test statistics presented in Sections \ref{sec assessing calibration} and \ref{sec assessing auto calibration}. That is, the test statistic for strong calibration $T_n^{\textrm{sc}}$ will be constructed using the weighted residuals
\begin{equation*}
%\label{weighted residuals sc}
    Z_i = (v_iY_i - v_i\widehat{\mu}_{\l}(\bX_i)) g_{\l, \B}(\bX_i), \quad i \in \TT,
\end{equation*}
whereas the test statistic for auto-calibration $T_n^{\textrm{ac}}$ will be constructed using the weighted residuals
\begin{equation}
    \label{weighted residuals ac}
    Z_i = (v_iY_i - v_i\widehat{\mu}_{\l}(\bX_i)) h_{\l, \B}(\widehat{\mu}_{\l}(\bX_i)), \quad i \in \TT.
\end{equation}
We finally point out that the definitions in Section \ref{sec cal auto cal} can be retrieved from \eqref{calibration v} and \eqref{auto-calibration v} by assuming, for example, that $v$ is independent of $\bX$ or $v = 1, \p\textrm{-a.s.}$

%We then additionally consider the isotonic recalibrated version of the three last models that we will denote by $\widehat{\mu}_{\l}^{Iso\, GLM} : \X \to \R$, $\widehat{\mu}_{\l}^{Iso\, GAM} : \X \to \R$ and $\widehat{\mu}_{\l}^{Iso\, NN} : \X \to \R$. \textit{Isotonic recalibration} is a method introduced by Wüthrich--Ziegel \cite{Wuethrich_Ziegel} in order to restore the auto-calibration of a regression function. It consists in fitting an isotonic regression on the responses by taking the regression function itself as a ranking function, leading to an empirically (i.e., on the level of the learning set $\l$) auto-calibrated regression function. Finally, as the true conditional mean is available in this example, we will also consider it in the next sections. This makes a total of 1

\subsection{Testing for strong calibration}

\label{sec num example test cal}

In order to assess the strong calibration of the above regression functions, we apply the procedure described in Section \ref{sec assessing calibration}. That is, we first fit a benchmark GBM using Poisson boosting trees on both the learning and boosting sets. For this, we follow Listing 7.2 in Wüthrich--Buser \cite{Wüthrich--Buser} and choose the same hyperparameters as for $\widehat{\mu}_{\l}^{\textrm{GBM1}} : \X \to (0, \infty)$, except for the learning rate that is now taken to be equal to $0.1$. Moreover, we select the optimal number of boosting iterations as above to prevent the model from overfitting. Table \ref{Tab_boost} reports this optimal number for both considered cases. It also shows the Poisson deviance losses $L(\cdot, \widehat{\mu}^\textrm{boost}_{\l, \B})$ attained by the benchmark GBM $\widehat{\mu}^{\mathrm{boost}}_{\ell,\mathcal{B}}:\mathcal{X}\to(0, \infty)$ on the learning and test sets, as well as the empirical Kullback-Leibler distance of the benchmark GBM with respect to the true conditional mean $\textrm{KL}_{\mu^{*}}(\cdot,\widehat{\mu}^\textrm{boost}_{\l, \B})$ evaluated on the test set. These metrics provide insight into the accuracy of the benchmark GBM and will be compared to the metrics of the eight models under consideration shown in Table \ref{Tab1}, below.

\begin{table}[htb!]
  \begin{center}
 {\small   
\begin{tabular}{|l||c|c|c|c|}
\hline
  Models & Boost.~steps & $L(\l, \widehat{\mu}^\textrm{boost}_{\l, \B})$ & $L(\T, \widehat{\mu}^\textrm{boost}_{\l, \B})$ & $\textrm{KL}_{\mu^{*}}(\T,\widehat{\mu}^\textrm{boost}_{\l, \B})$ \\
  \hline\hline
{\it Case 1.} $\mathcal{L} : 60\%$, $\B : 20\%$, $\T : 20\%$ & -&- &- &- \\
\hline
Benchmark GBM & 217 & 27.704 & 27.503 & 0.072 \\
 \hline \hline
{\it Case 2.} $\mathcal{L} : 80\%$, $\B : 10\%$, $\T : 10\%$ &- &- &- &-\\
\hline
Benchmark GBM & 209 & 27.764 & 27.859 & 0.067 \\
\hline
\end{tabular}}
\end{center}
\caption{Number of boosting steps used to fit the benchmark GBM for each case. The Poisson deviance losses and the empirical Kullback-Leibler distances are reported in $10^{-2}$.}
\label{Tab_boost}
\end{table}
The benchmark GBM is then used to construct a test function $g_{\l,  \B}:\X\to\R$ with
\[
g_{\ell,\mathcal{B}}(\bx)=
\begin{cases}
1, & \text{if } \widehat{\mu}^{\mathrm{boost}}_{\ell,\mathcal{B}}(\bx) > \widehat{\mu}_{\ell}(\bx),
\\[0.2cm]

0.01,
& \text{if }
\widehat{\mu}^{\mathrm{boost}}_{\ell,\mathcal{B}}(\bx)
=
\widehat{\mu}_{\ell}(\bx),
\\[0.2cm]

-1,
& \text{if }
\widehat{\mu}^{\mathrm{boost}}_{\ell,\mathcal{B}}(\bx)
<
\widehat{\mu}_{\ell}(\bx).
\end{cases}
\]
This test function satisfies the requirements in Section \ref{sec assessing calibration}, which, as previously discussed, are meant to push the test statistic $T_n^{\textrm{sc}}$ away from $0$ whenever the null-hypothesis of strong calibration is violated. In our construction, we choose the test function to take values in the set $\{-1,0.01,1\}$ depending on the sign of the biases detected by the benchmark GBM, but we emphasize that, in principle, alternative constructions for the test function could also be considered. For example, one could take into account the size of the detected biases as well.

%\bigskip
To better understand the role of the test function $g_{\l, B} : \X \to \R$ for each considered model, we plot the mean estimates against the true conditional means on the test set $\T$ in Figure \ref{new_Fig1}. Points for which $g_{\l,\B} = 1$ are shown in black, those for which $g_{\l,\B} = 0.01$ in green, and those for which $g_{\l,\B} = -1$ in red. We restrict here attention to the case where the learning set comprises 60\% of the data, although similar plots can be obtained for the other case as well.  

\begin{center}
\begin{tabular}{cc}
\includegraphics[width=.42\textwidth]{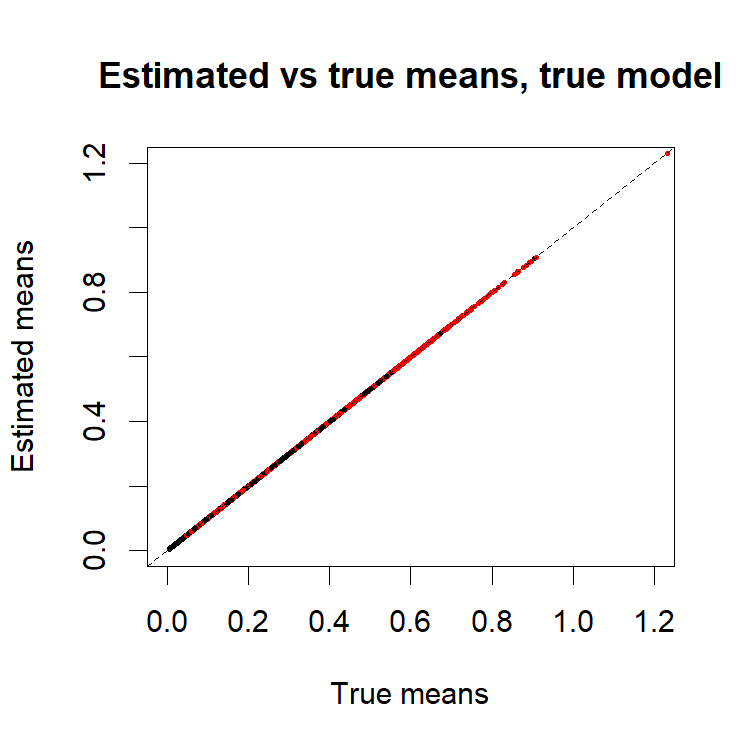} &
\includegraphics[width=.42\textwidth]{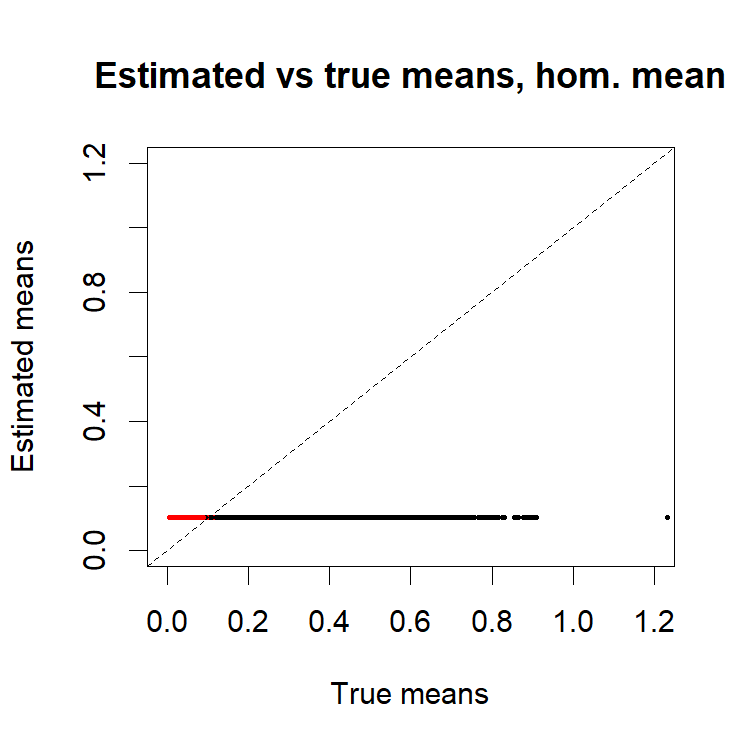} \\[0.3cm]

\end{tabular}
\end{center}

\begin{center}
\begin{tabular}{cc}

\includegraphics[width=.42\textwidth]{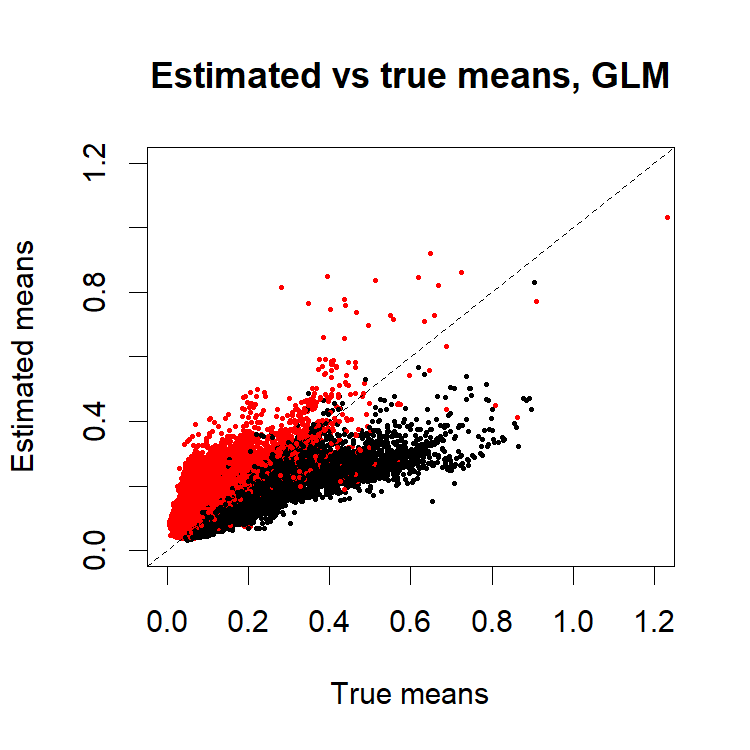} &
\includegraphics[width=.42\textwidth]{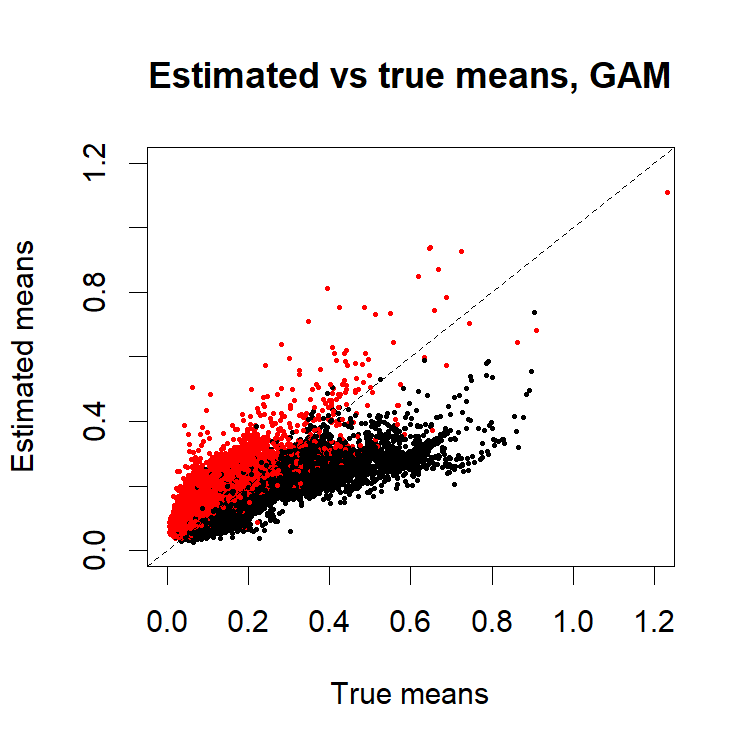} \\[0.3cm]

\includegraphics[width=.42\textwidth]{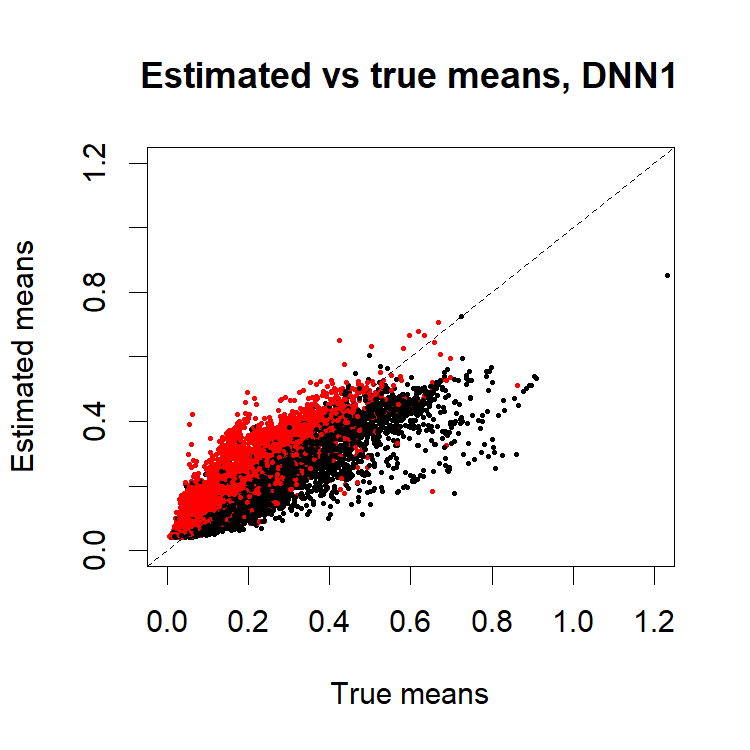} &
\includegraphics[width=.42\textwidth]{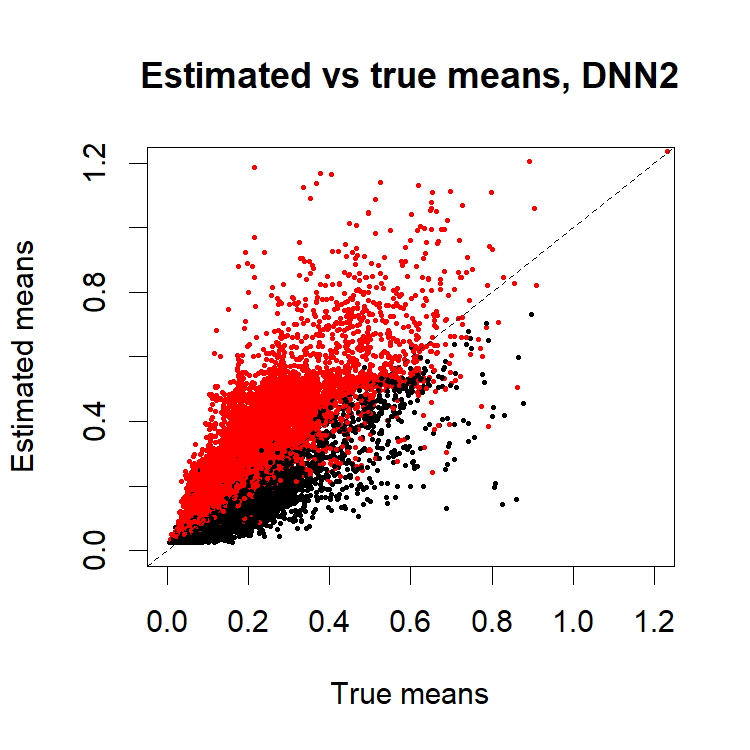}
\end{tabular}
\end{center}

\begin{center}
\begin{minipage}{0.42\textwidth}
\centering
\includegraphics[width=\linewidth]{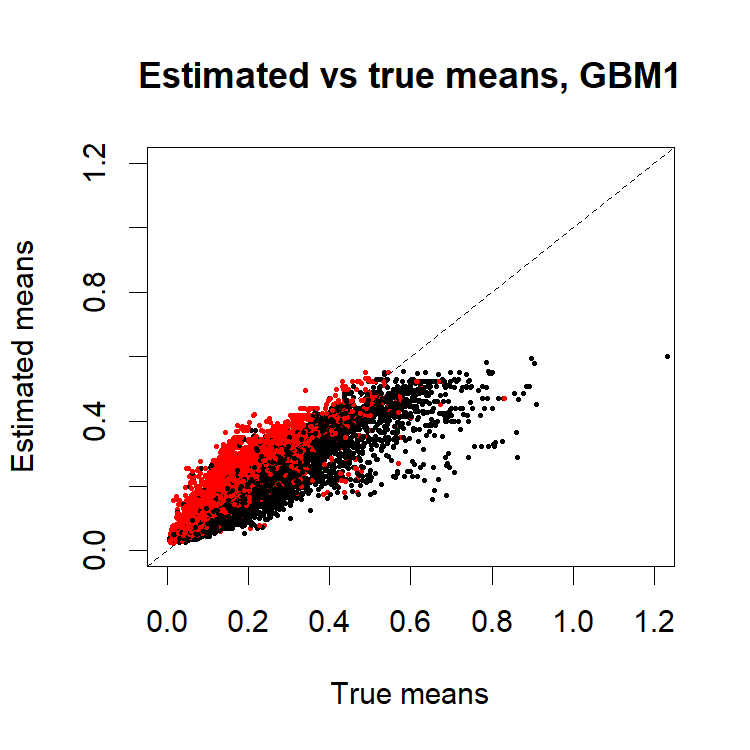}
\end{minipage}
\hspace{0.03\textwidth}
\begin{minipage}{0.42\textwidth}
\centering
\includegraphics[width=\linewidth]{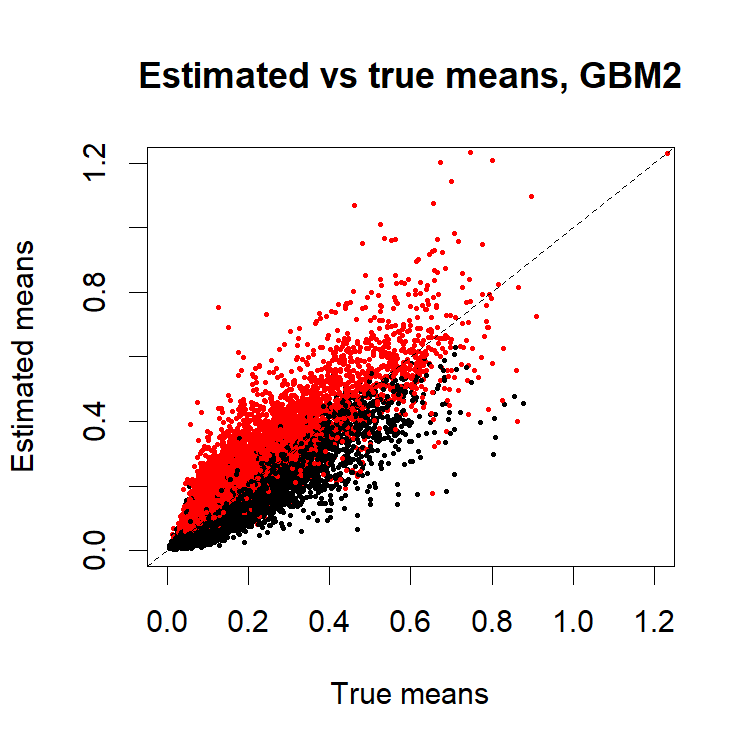}
\end{minipage}

\captionof{figure}{
Plots of the mean estimates against the true conditional means on the test set $\T$ for eight different models. Only the case where the size of the learning set corresponds to $60\%$ of the dataset $\D$ is considered. Points for which $g_{\l,\B} = 1$ are shown in black, those for which $g_{\l,\B} = 0.01$ in green, and those for which $g_{\l,\B} = -1$ in red.
}
\label{new_Fig1}

\end{center}

The first observation is that, with the exception of the true conditional mean, the benchmark GBM $\widehat{\mu}^\textrm{boost}_{\l, \B} : \X \to (0,\infty)$ successfully identifies most positive and negative biases on the test set $\T$. Indeed, the majority of black points lie below the diagonal, whereas most red points lie above it. Nevertheless, a few red points appear below the diagonal and a few black points above it, which implies that for some sets of features $\bx \in \X$, the differences
$$
\widehat{\mu}^\textrm{boost}_{\l, \B}(\bx) -\widehat{\mu}_\l(\bx) \quad \textrm{and} \quad \mu^{*}(\bx) -\widehat{\mu}_\l(\bx)
$$
have opposite signs. Additionally, note that the benchmark GBM detects positive and negative biases for the true model on the learning and boosting sets although it is strongly calibrated. This simply indicates that the benchmark GBM does not coincide with the true conditional mean. Interestingly, we further notice that the value $g_{\l, \B} = 0.01$ is never taken for any of the considered models, see Remark \ref{rem theory cal}.

In fact, the plots in Figure \ref{new_Fig1}, that would not be available in practice because the true conditional means would be unknown, provide a graphical assessment of strong calibration for each model. For instance, a model is strongly calibrated if all points lie on the diagonal, and the more points are close to it, the more accurate the model is. According to Figure \ref{new_Fig1}, the best fits are provided by the models DNN1 and GBM1, for which the underlying points lie closest to the diagonal among all models. For the models DNN2 and GBM2, we see that the points are centered around the diagonal, but their spread is considerably larger. This is a typical feature of overfitting models, for which large true conditional means may be associated with small estimated means, and vice versa. In contrast, the models GLM and GAM exhibit a level of variability similar to that of the models DNN1 and GBM1. We observe, however, a bend on the right-hand side of the plots, indicating that large true conditional means are assigned to small mean estimates. The homogeneous model, for its part, provides the worst fit as the corresponding points are not even clustered around the diagonal. 

Interestingly, although the plots in Figure~\ref{new_Fig1} would not be available in practice, a similar analysis could be performed by comparing the values taken by the initially fitted regression functions with those of the benchmark GBM. This could help identify potential systematic biases and provide insights for improving the original models. However, two important limitations to bear in mind are that the benchmark GBM is fitted using more data, namely the boosting set, and that it does not coincide with the true conditional mean. The resulting conclusions should thus be interpreted with caution.

Next, our goal is to perform strong calibration tests to understand whether the violations of strong calibration in Figure \ref{new_Fig1} are statistically significant. The results of the performed tests are given in Table \ref{Tab1}, where we additionally provide the Poisson deviance losses of all considered models on the learning and test sets, as well as the empirical Kullback-Leibler distance of these models with respect to the true conditional mean evaluated on the test set. 

\begin{table}[htb!]
  \begin{center}
 {\small   
\begin{tabular}{|l||c|c|c|c|}
\hline
  Models & $p$-values & $L(\l, \widehat{\mu}_\ell)$ & $L(\T, \widehat{\mu}_\ell)$ & $\textrm{KL}_{\mu^{*}}(\T,\widehat{\mu}_\ell)$ \\
  \hline\hline
{\it Case 1.} $\mathcal L : 60\%$, $\B : 20\%$, $\T : 20\%$ & - & - & - & - \\
\hline
True model & $1.54 \cdot 10^{-1}$ & 27.778 & 27.333 & 0 \\
Homogeneous mean & $2.13 \cdot 10^{-140}$ & 29.168 & 28.706 & 0.671 \\
GLM & $3.62 \cdot 10^{-52}$ & 28.229 & 27.877 & 0.242 \\
GAM & $1.73 \cdot 10^{-31}$ & 28.180 & 27.788 & 0.218 \\
DNN1 & $1.30 \cdot 10^{-17}$ & 27.941 & 27.633 & 0.141 \\
DNN2 & $8.23 \cdot 10^{-46}$ & 27.842 & 27.845 & 0.246 \\
GBM1 & $1.02 \cdot 10^{-18}$ & 27.850 & 27.609 & 0.118 \\
GBM2 & $3.62 \cdot 10^{-38}$ & 26.846 & 27.776 & 0.234 \\
 \hline \hline
{\it Case 2.} $\mathcal L : 80\%$, $\B : 10\%$, $\T : 10\%$ & - & - & - & - \\
\hline
True model & $5.58 \cdot 10^{-2}$ & 27.827 & 27.677 & 0 \\
Homogeneous mean & $6.24 \cdot 10^{-67}$ & 29.207 & 29.043 & 0.688 \\
GLM & $2.08 \cdot 10^{-24}$ & 28.276 & 28.230 & 0.247 \\
GAM & $5.62 \cdot 10^{-17}$ & 28.230 & 28.124 & 0.222 \\
DNN1 & $3.24 \cdot 10^{-7}$ & 27.907 & 27.960 & 0.128 \\
DNN2 & $1.39 \cdot 10^{-8}$ & 27.741 & 27.959 & 0.170 \\
GBM1 & $1.80 \cdot 10^{-4}$ & 27.823 & 27.849 & 0.083 \\
GBM2 & $7.85 \cdot 10^{-23}$ & 27.018 & 28.140 & 0.198 \\
\hline
\end{tabular}}
\end{center}
\caption{$p$-values of the strong calibration test for eight different models and two different cases. The out-of-sample Poisson deviance losses as well as the empirical Kullback--Leibler distances are reported in $10^{-2}$.}
\label{Tab1}
\end{table}
For both cases, our test rejects the strong calibration of all models at a confidence level of $1-\alpha = 0.95$, except for the true conditional mean. Moreover, we observe that the ranking of the accuracy of the models provided by the empirical Kullback-Leibler distance $\textrm{KL}_{\mu^{*}}(\T,\widehat{\mu}_\ell)$ matches the ranking induced by the underlying $p$-values, with only a few exceptions. Since this measure would not be available in practice, one may instead rely on the out-of-sample Poisson deviance loss $L(\T, \widehat{\mu}_\ell)$ to compare models and the resulting ranking remains largely consistent with the ranking obtained from the $p$-values, differing only in a few cases. 
Although the $p$-values
reported in Table~\ref{Tab1} should not be interpreted as measures of
predictive accuracy or, in general, be used to rank models according to
their predictive performance, the above observations suggest that, in this example, the
proposed test performs similarly across the different types of regression
models considered. We further notice that the $p$-values for the second case are systematically higher than for the first case, except for the true model. A possible explanation is that the size of the test set $\T$ is smaller in the second case and models fitted on $80 \%$ of the data generally achieve a better fit. 

This example shows that the proposed testing procedure in Section \ref{sec assessing calibration} leads to statistical tests that are able to detect violations of strong calibration for all the fitted models. This is actually not surprising as the metrics in Table \ref{Tab_boost} show that the benchmark GBMs exhibit a higher predictive performance than all the other considered models, except for the true conditional mean. As pointed out at the beginning of this section, we emphasize that more complex test functions could in principle be used for the test decision. We refrain from doing so here as the performed tests already provide satisfactory results by only taking into account the sign of the biases between the benchmark GBMs and the originally fitted regression functions. Finally, note that a comparison with the graphical strong calibration test of Delong et al.~\cite{Delong} is provided in Appendix \ref{appendix cal bands}.

\subsection{Testing for auto-calibration}

\label{sec num example test auto cal}

As auto-calibration is a weaker property than strong calibration, we might expect that some of the models under consideration are auto-calibrated, but not strongly calibrated. In fact, we know that the true conditional mean is auto-calibrated as it is strongly calibrated, whereas the homogeneous mean $\widehat{\mu}_\l^{\textrm{hom}} : \X \to (0, \infty)$ is {\it empirically auto-calibrated}, i.e., for all $m \in \widehat{\mu}_\l^{\textrm{hom}}(\X)$,
$$
\sum_{i \in \LL} v_i \widehat{\mu}_\l^{\textrm{hom}}(\bx_i) \mathds{1}_{\{\widehat{\mu}_\l^{\textrm{hom}}(\bx_i) = m\}} = \sum_{i \in \LL} v_i Y_i \mathds{1}_{\{\widehat{\mu}_\l^{\textrm{hom}}(\bx_i) = m\}}, $$
see \eqref{hom mean}. Empirical auto-calibration is an in-sample property on the learning set $\l$ and does not necessarily imply auto-calibration for a general pair $(Y,\bX,v)$, see \eqref{auto-calibration v}. It is therefore important to assess whether this property also holds beyond the learning sample. To this end, we adapt, as above, the general testing procedure for auto-calibration in Section \ref{sec assessing auto calibration} by using the weighted residuals in \eqref{weighted residuals ac} to define the test statistic for auto-calibration $T_n^{\textrm{ac}}$. Moreover, we modify the definition of the map $\widetilde{\mu}_{\l, \B}:\widehat{\mu}_{\l}(\X)\to\R$ to take into account the exposures. That is, \eqref{impossible map} becomes
\begin{equation}
        \label{emp weighted avg}
  \widetilde{\mu}^\textrm{boost}_{\l, \B}(m) = \frac{\sum_{i \in {\LL \cup \BB}} v_i \widehat{\mu}^\textrm{boost}_{\ell, \B}(\bx_i) \mathds{1}_{\{\widehat{\mu}_\l(\bx_i) = m\}}}{\sum_{i \in {\LL \cup \BB}} v_i \mathds{1}_{\{\widehat{\mu}_\l(\bx_i) = m\}}},
\end{equation}
for $m \in \widehat{\mu}_{\l}(\X_\ell \cup \X_\B)$, and we extend the domain of this map to the whole mean estimate space in the same way as in Section \ref{sec assessing auto calibration}. Finally, we construct a test function $h_{\ell, \B}: \widehat{\mu}_{\l}(\X) \to \R$ with
\[
h_{\ell,\mathcal{B}}(m)=
\begin{cases}
1, & \text{if } \widetilde{\mu}^{\mathrm{boost}}_{\ell,\mathcal{B}}(m) > m,
\\[0.2cm]

0.01,
& \text{if }
\widetilde{\mu}^{\mathrm{boost}}_{\ell,\mathcal{B}}(m) = m,
\\[0.2cm]

-1,
& \text{if }
\widetilde{\mu}^{\mathrm{boost}}_{\ell,\mathcal{B}}(m) < m.
\end{cases}
\]

\begin{figure}[b!]
\centering

\begin{minipage}{0.42\textwidth}
\centering
\includegraphics[width=\linewidth]{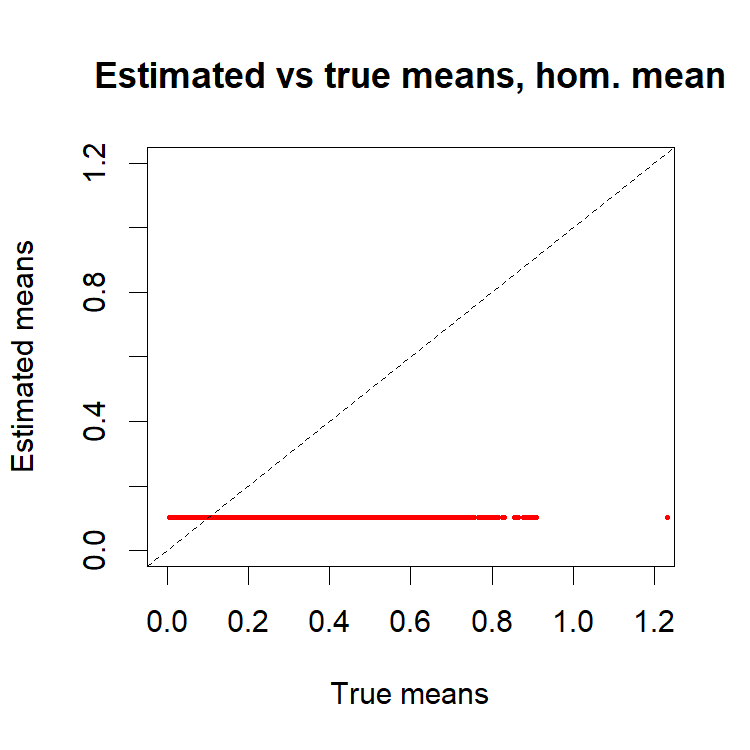}
\end{minipage}
\hspace{0.03\textwidth}
\begin{minipage}{0.42\textwidth}
\centering
\includegraphics[width=\linewidth]{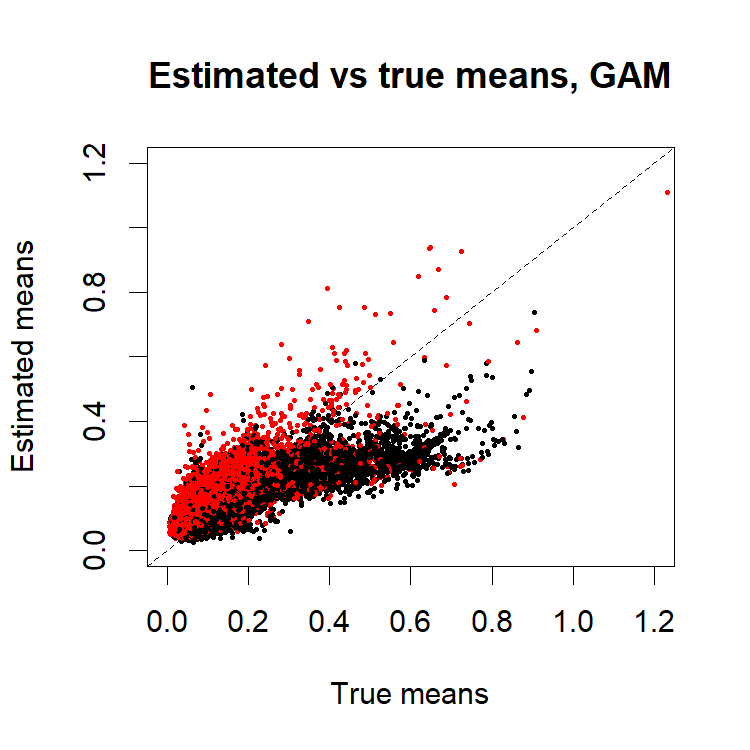}
\end{minipage}

\captionof{figure}{
Plots of the mean estimates against the true conditional means on the test set $\T$ for two different models. Only the case where the size of the learning set corresponds to $60\%$ of the dataset $\D$ is considered. Points for which $h_{\l, B} = 1$ are shown in black, those for which $h_{\l,\B} = 0.01$ in green, and those for which $h_{\l,\B} = -1$ in red. 
}
\label{new_Fig2}

\end{figure}

Before presenting the results of the performed tests, we plot the mean estimates against the true means for two different models in Figure \ref{new_Fig2}. These plots should be compared to those in Figure \ref{new_Fig1} as we use the same colours to highlight the values taken by the test function. For both models, we observe that only the colour of the points differs from Figure \ref{new_Fig1}. The reason is that, now, all the dots lying on the same horizontal line have to be of the same colour as the test function $h_{\l,B} : \widehat{\mu}_\ell(\X) \to \R$ is defined on the mean estimate space. This change is clearly noticeable for the homogeneous mean, for which all the points became red. This indicates that the single value taken by the homogeneous mean $\widehat{\mu}_\l : \X \to (0, \infty)$ is larger than the weighted average of the benchmark GBM $\widehat{\mu}^\textrm{boost}_{\l, \B} : \X  \to (0, \infty)$ on the learning and boosting sets, see \eqref{emp weighted avg}. As the GAM is much more granular, the change with respect to the previous section is less easy to recognize in the right plot of Figure \ref{new_Fig2}. A closer look reveals, however, that more red points now lie below the diagonal, while more black points lie above it.

The results of the performed auto-calibration tests are provided in Table \ref{Tab2}, where we compare them to the results of the test of Denuit et al.~\cite{Denuit2}. The latter test is performed using $B = 500$ Monte Carlo simulations on the set $\D \setminus \mathcal{L}$ as it does not require to use any boosting set $\B$.

\begin{table}[htb!]
  \begin{center}
 {\small   
\begin{tabular}{|l||c|c|}
\hline
  Models
  & \begin{tabular}{c}
    $p$-values of the\\
    test of Denuit et al.~\cite{Denuit2}
    \end{tabular}
  & \begin{tabular}{c}
    $p$-values of the\\
    test of Section \ref{sec assessing auto calibration}
    \end{tabular}\\
\hline\hline

{\it Case 1.} $\mathcal L : 60\%$, $\B : 20\%$, $\T : 20\%$
& - & -\\
\hline
True model       & $5.94 \cdot 10^{-1}$ & $9.32 \cdot 10^{-1}$\\
Homogeneous mean & $2.98 \cdot 10^{-1}$ & $2.06 \cdot 10^{-1}$\\
GLM              & $4.72 \cdot 10^{-1}$ & $3.99 \cdot 10^{-33}$\\
GAM              & $6.20 \cdot 10^{-2}$ & $1.02 \cdot 10^{-14}$\\
DNN1             & $0$                  & $1.34 \cdot 10^{-10}$\\
DNN2             & $0$                  & $1.34 \cdot 10^{-33}$\\
GBM1             & $2.42 \cdot 10^{-1}$ & $2.22 \cdot 10^{-11}$\\
GBM2             & $0$                  & $7.04 \cdot 10^{-25}$\\

\hline\hline
{\it Case 2.} $\mathcal L : 80\%$, $\B : 10\%$, $\T : 10\%$
& - & -\\
\hline
True model       & $3.78 \cdot 10^{-1}$ & $8.74 \cdot 10^{-1}$\\
Homogeneous mean & $2.18 \cdot 10^{-1}$ & $7.90 \cdot 10^{-1}$\\
GLM              & $2.78 \cdot 10^{-1}$ & $6.48 \cdot 10^{-15}$\\
GAM              & $1.80 \cdot 10^{-2}$ & $3.61 \cdot 10^{-8}$\\
DNN1             & $5.80 \cdot 10^{-2}$ & $1.32 \cdot 10^{-2}$\\
DNN2             & $6.00 \cdot 10^{-3}$ & $5.10 \cdot 10^{-5}$\\
GBM1             & $1.70 \cdot 10^{-1}$ & $4.44 \cdot 10^{-4}$\\
GBM2             & $0$                 & $8.54 \cdot 10^{-14}$\\
\hline
\end{tabular}}
\end{center}
\caption{$p$-values of the auto-calibration tests for eight different models and two different cases. The auto-calibration test of Section \ref{sec assessing auto calibration} is performed on the test set $\T$, whereas the test of Denuit et al.~\cite{Denuit2} is performed using 500 Monte Carlo simulations on the set $\D \setminus\mathcal{L}$.}
\label{Tab2}
\end{table}

Our test rejects the auto-calibration of all considered models at a confidence level of $1-\alpha = 0.95$, except for the true conditional mean and the homogeneous mean, whereas the test of Denuit et al.~\cite{Denuit2} rejects the null-hypothesis of auto-calibration for fewer models. In particular, there are no models for which auto-calibration gets rejected under the test of Denuit et al.~\cite{Denuit2}, but not under our test. Additionally, we observe that the $p$-values of our test in Table \ref{Tab2} are all of a similar magnitude or higher than the $p$-values in Table \ref{Tab1}.%, and this is not surprising because calibration is a stronger property than auto-calibration. 

In order to confirm the results induced by the above tests, we provide in Figure \ref{new_Fig3} the actual vs.~predicted plots of the eight models under consideration for the case where the learning set corresponds to 60\% of the data. In these plots, responses are first binned according to the deciles of the mean estimates given by $\widehat{\mu}_\l : \X \to (0,\infty)$. These bins are shown as rectangles below. Then, the empirical mean of the responses (actuals) is plotted against the average mean estimate in blue for each bin. 

\begin{center}

\begin{minipage}{0.42\textwidth}
\centering
\includegraphics[width=\linewidth]{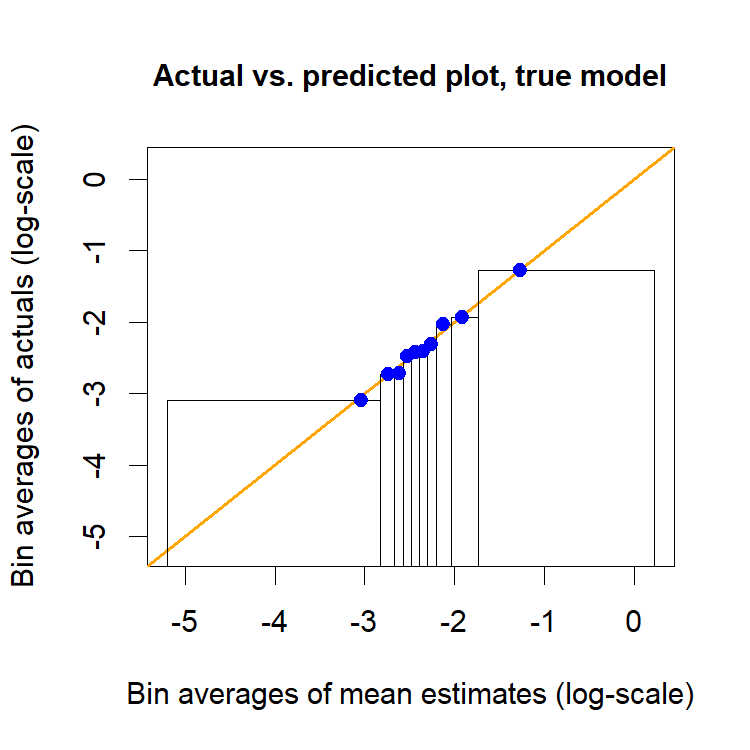}
\end{minipage}
\hspace{0.03\textwidth}
\begin{minipage}{0.42\textwidth}
\centering
\includegraphics[width=\linewidth]{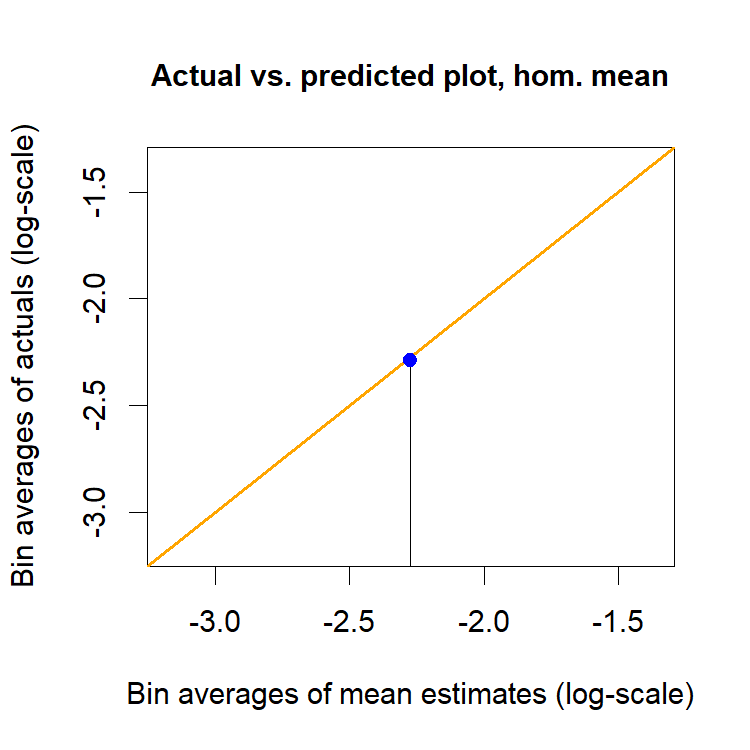}
\end{minipage}

\vspace{0.3cm}

\begin{minipage}{0.42\textwidth}
\centering
\includegraphics[width=\linewidth]{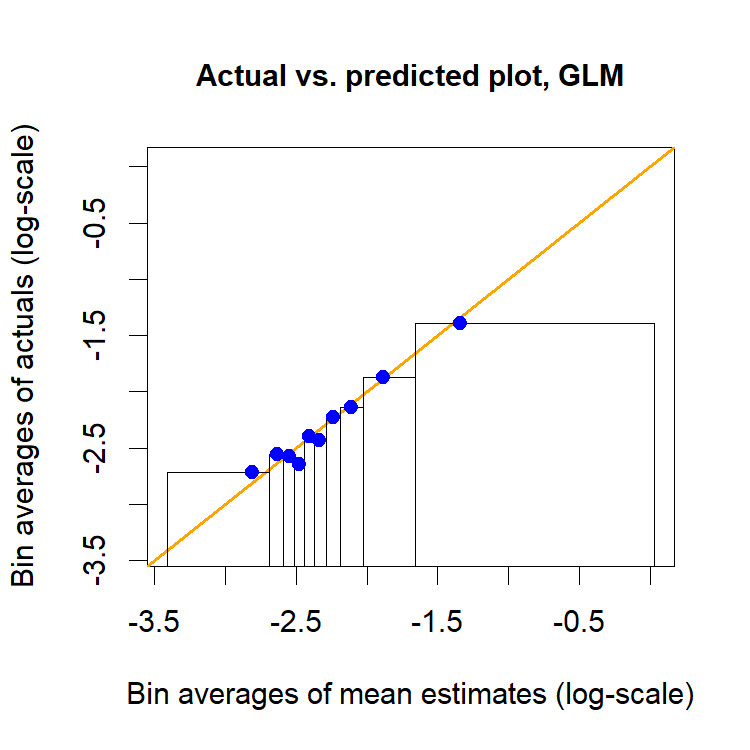}
\end{minipage}
\hspace{0.03\textwidth}
\begin{minipage}{0.42\textwidth}
\centering
\includegraphics[width=\linewidth]{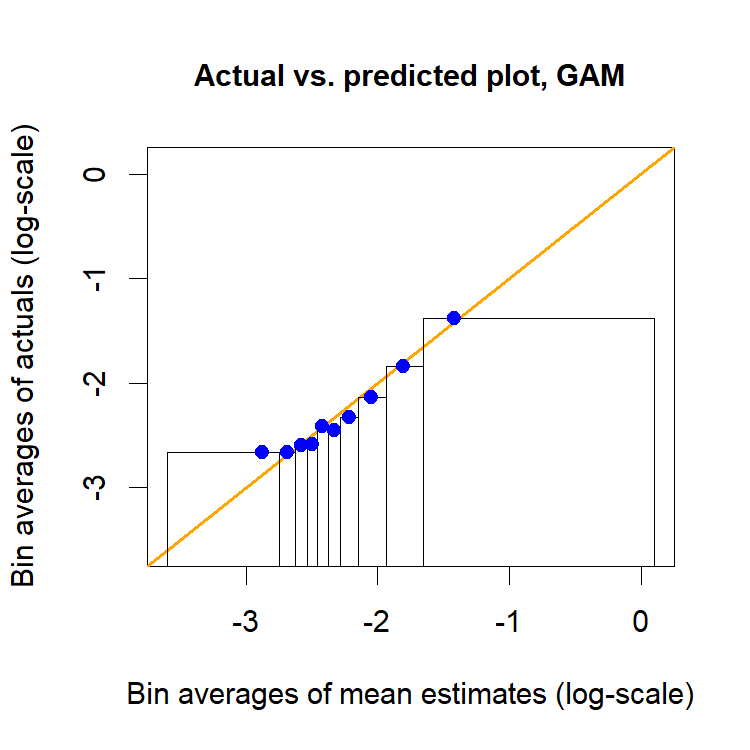}
\end{minipage}

\vspace{0.3cm}

\begin{minipage}{0.42\textwidth}
\centering
\includegraphics[width=\linewidth]{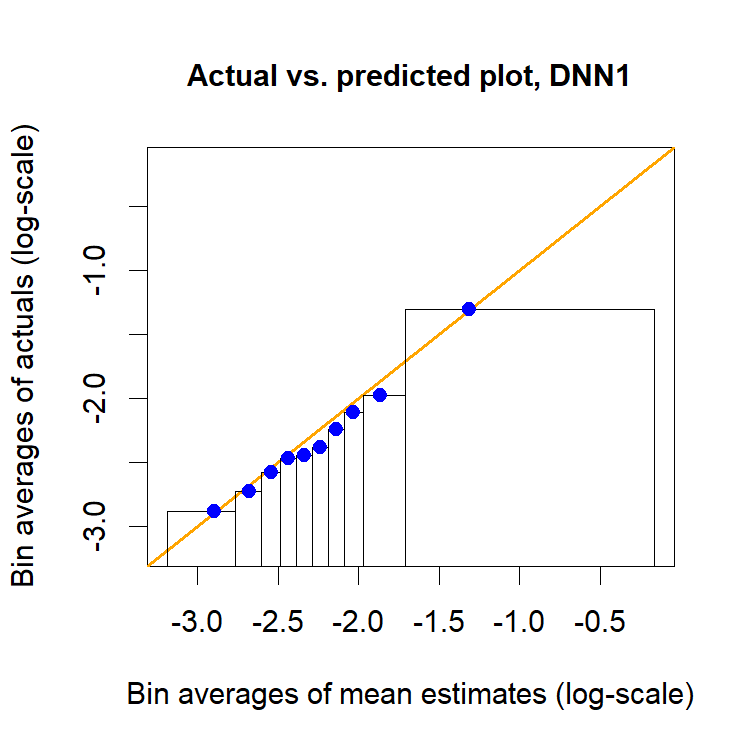}
\end{minipage}
\hspace{0.03\textwidth}
\begin{minipage}{0.42\textwidth}
\centering
\includegraphics[width=\linewidth]{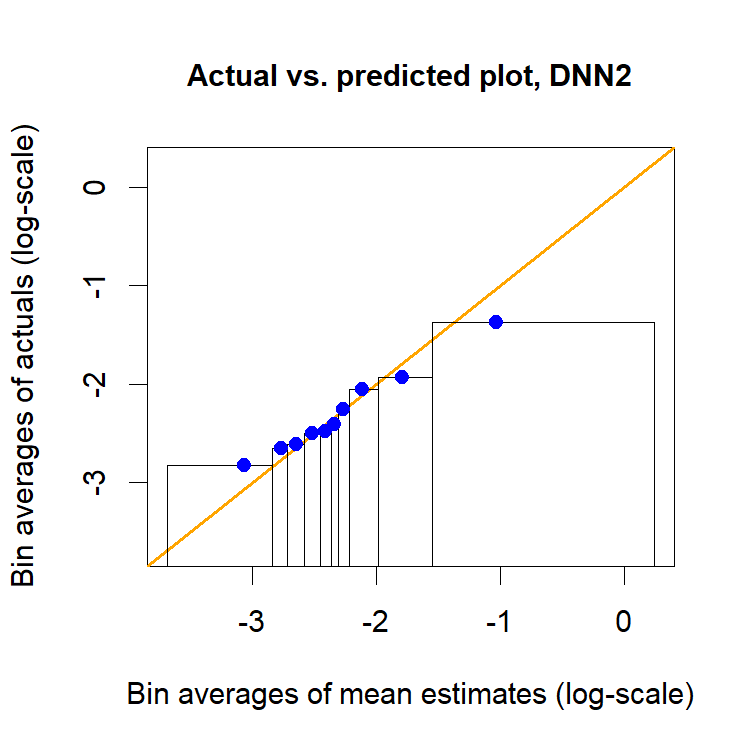}
\end{minipage}

\vspace{0.3cm}

\begin{minipage}{0.42\textwidth}
\centering
\includegraphics[width=\linewidth]{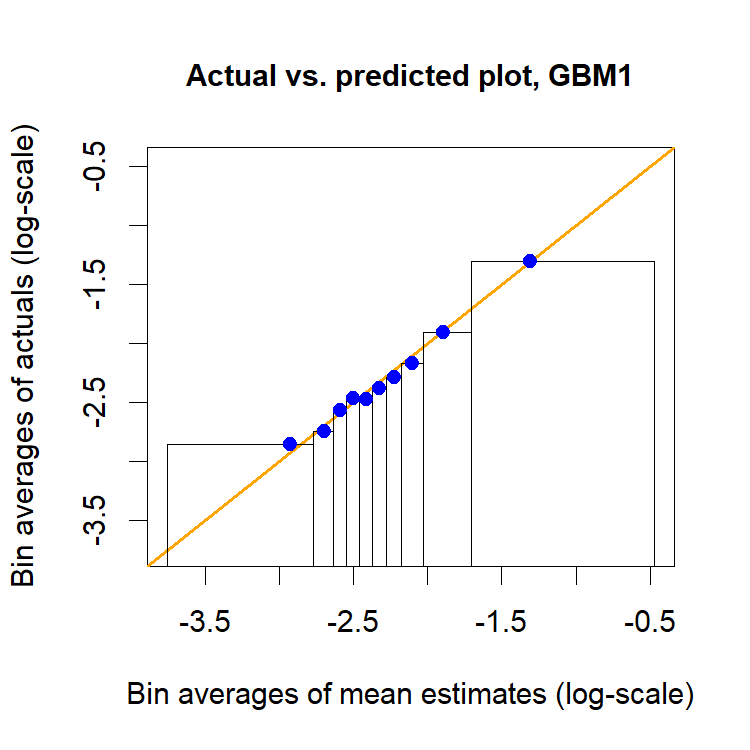}
\end{minipage}
\hspace{0.03\textwidth}
\begin{minipage}{0.42\textwidth}
\centering
\includegraphics[width=\linewidth]{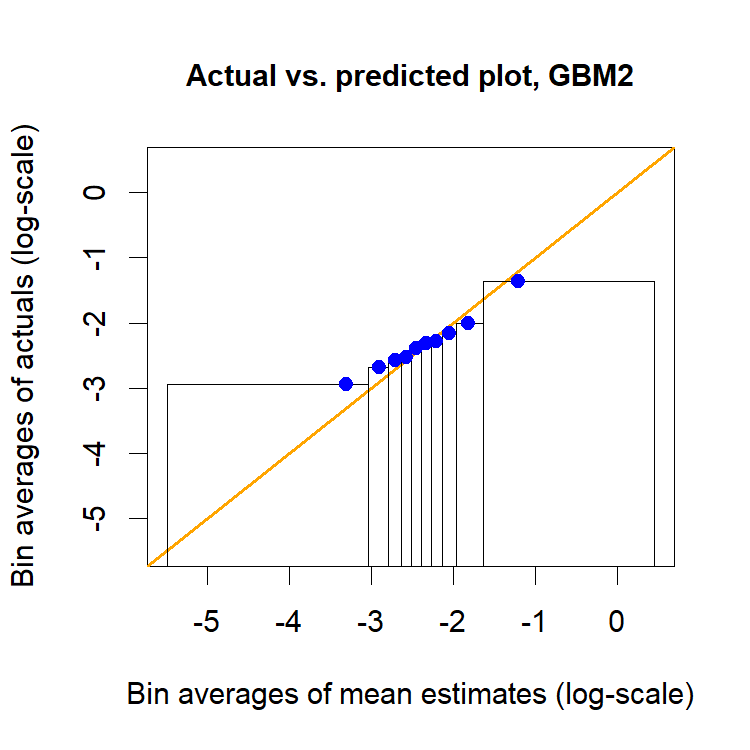}
\end{minipage}

\captionof{figure}{
Plots of binned averages of responses (actuals) against binned averages of mean estimates on the test set $\T$ for eight different models. Only the case where the size of the learning set corresponds to $60\%$ of the dataset $\D$ is considered. Bin intervals are shown as rectangles.
}
\label{new_Fig3}

\end{center}

A model violates auto-calibration whenever the corresponding blue points in Figure \ref{new_Fig3} deviate significantly from the diagonal. Such deviations occur for nearly all models, with the exception of the true model, the homogeneous mean, and the model GBM1. The actual vs.~predicted plots thus corroborate the conclusions of our test, except for the model GBM1. A possible reason for this is that violations of auto-calibration for the latter model may occur more locally, within the bins themselves.

Next, we consider the isotonic recalibrated versions of the models GLM, GAM, DNN1, DNN2, GBM1 and GBM2. \textit{Isotonic recalibration} is a method used in Wüthrich--Ziegel \cite{Wuethrich_Ziegel} in order to restore the auto-calibration of a given regression function $\widehat{\mu}_\l : \X \to \R$. It consists in ranking the realizations of the learning set $\l = (y_i, \bx_i, v_i)_{i \in \LL}$ such that the mean estimates are non-decreasing, i.e., such that
\begin{equation}
    \label{isotonic ranking}
    \widehat{\mu}_\l(\bx_1) \leq \widehat{\mu}_\l(\bx_2) \leq \dots \leq \widehat{\mu}_\l(\bx_{|\mathcal{L}|-1}) \leq \widehat{\mu}_\l(\bx_{|\mathcal{L}|}), 
\end{equation}
and fitting an isotonic regression on the underlying responses. This leads to a vector 
\begin{equation*}    
    %\label{isotonic estimator}
    \widehat{\bmu}_\l^{\textrm{Iso}} = \mathop{\arg \min}\limits_{\bmu \in \R^n} \Big\{ \sum_{i=1}^n v_i(y_i-\mu_i)^2: \mu_1 \leq \dots \leq \mu_n \Big\},
\end{equation*}
providing recalibrated mean estimates for $(\bx_i)_{i=1}^{|\mathcal{L}|}$.
%as the i-th component of the vector
%\begin{equation*}    
    %\label{isotonic estimator}
    %\widehat{\bmu}_\l^{\textrm{Iso}} = \mathop{\arg \min}\limits_{\bmu \in \R^n} \Big\{ \sum_{i=1}^n v_i(y_i-\mu_i)^2: \mu_1 \leq \dots \leq \mu_n \Big\}.
%\end{equation*}
The extension to the whole feature space $\X$ is then done using a step function interpolation; we refer to Section 2 in Wüthrich--Ziegel \cite{Wuethrich_Ziegel} for an extended description about this method. As isotonic recalibration leads to empirically auto-calibrated regression functions by construction, see Wüthrich--Ziegel \cite{Wuethrich_Ziegel}, our aim is to determine whether this property also holds on out-of-sample data, i.e., for a general pair $(Y, \bX, v)$. The results of the performed tests at a confidence level of $1-\alpha = 0.95$ are given in Table \ref{Tab3}, where a comparison with the test of Denuit et al.~\cite{Denuit2} is provided again.

%It consists in fitting an isotonic regression on the responses by taking the regression function itself as a ranking function. This leads to an empirically auto-calibrated regression function on $\l$. The results of our test as well as the test of Denuit et al.~\cite{Denuit} are provided in Table \ref{Tab3}.

\begin{table}[htb!]
  \begin{center}
 {\small   
\begin{tabular}{|l||c|c|}
\hline
  Models
  & \begin{tabular}{c}
    $p$-values of the\\
    test of Denuit et al.~\cite{Denuit2}
    \end{tabular}
  & \begin{tabular}{c}
    $p$-values of the\\
    test of Section \ref{sec assessing auto calibration}
    \end{tabular}\\
\hline\hline

{\it Case 1.} $\mathcal L : 60\%$, $\B : 20\%$, $\T : 20\%$
& - & -\\
\hline
Recalibrated GLM  & $5.84 \cdot 10^{-1}$ & $7.98 \cdot 10^{-2}$\\
Recalibrated GAM  & $5.18 \cdot 10^{-1}$ & $1.43 \cdot 10^{-3}$\\
Recalibrated DNN1 & $6.36 \cdot 10^{-1}$ & $3.50 \cdot 10^{-1}$\\
Recalibrated DNN2 & $2.80 \cdot 10^{-2}$ & $1.25 \cdot 10^{-6}$\\
Recalibrated GBM1 & $9.60 \cdot 10^{-2}$ & $7.12 \cdot 10^{-3}$\\
Recalibrated GBM2 & $0$                  & $1.30 \cdot 10^{-56}$\\

\hline\hline
{\it Case 2.} $\mathcal L : 80\%$, $\B : 10\%$, $\T : 10\%$
& - & -\\
\hline
Recalibrated GLM  & $3.96 \cdot 10^{-1}$ & $4.09 \cdot 10^{-3}$\\
Recalibrated GAM  & $2.22 \cdot 10^{-1}$ & $9.68 \cdot 10^{-3}$\\
Recalibrated DNN1 & $2.86 \cdot 10^{-1}$ & $7.61 \cdot 10^{-3}$\\
Recalibrated DNN2 & $1.52 \cdot 10^{-1}$ & $4.10 \cdot 10^{-3}$\\
Recalibrated GBM1 & $3.06 \cdot 10^{-1}$ & $1.22 \cdot 10^{-1}$\\
Recalibrated GBM2 & $0$                  & $6.54 \cdot 10^{-22}$\\
\hline
\end{tabular}}
\end{center}
\caption{$p$-values of the  auto-calibration tests for six different isotonic recalibrated models and two different cases. The auto-calibration test of Section \ref{sec assessing auto calibration} is performed on the test set $\T$, whereas the test of Denuit et al.~\cite{Denuit2} is performed using 500 Monte Carlo simulations on the set $\D \setminus\mathcal{L}$.}
\label{Tab3}
\end{table}

For the six new isotonic recalibrated models under consideration, only the auto-calibration of three models does not get rejected by our test at a confidence level of $1-\alpha = 0.95$. These models are the isotonic recalibrated GLM and the isotonic recalibrated DNN1 fitted on $60 \%$ of the data as well as the isotonic recalibrated GBM1 fitted on $80 \%$ of the data. We emphasize, however, that the resulting $p$-values are much higher than in Table \ref{Tab2}, hinting that isotonic recalibration helps to partly restore auto-calibration for the initially fitted regression functions outside the learning set. Interestingly, those $p$-values remain low for the recalibrated overfitting models DNN2 and GBM2. This is due to a ranking of the means inferred from the learning set in \eqref{isotonic ranking} that is too different from the ranking given by the true conditional means, causing the empirical auto-calibration property to fail on the test set. Moreover, we observe that, once again, the test of Denuit et al.~\cite{Denuit2} rejects the null hypothesis of auto-calibration for fewer isotonically recalibrated models than our proposed test.

\begin{figure}[htb!]
\centering
\begin{minipage}{0.42\textwidth}
\centering
\includegraphics[width=\linewidth]{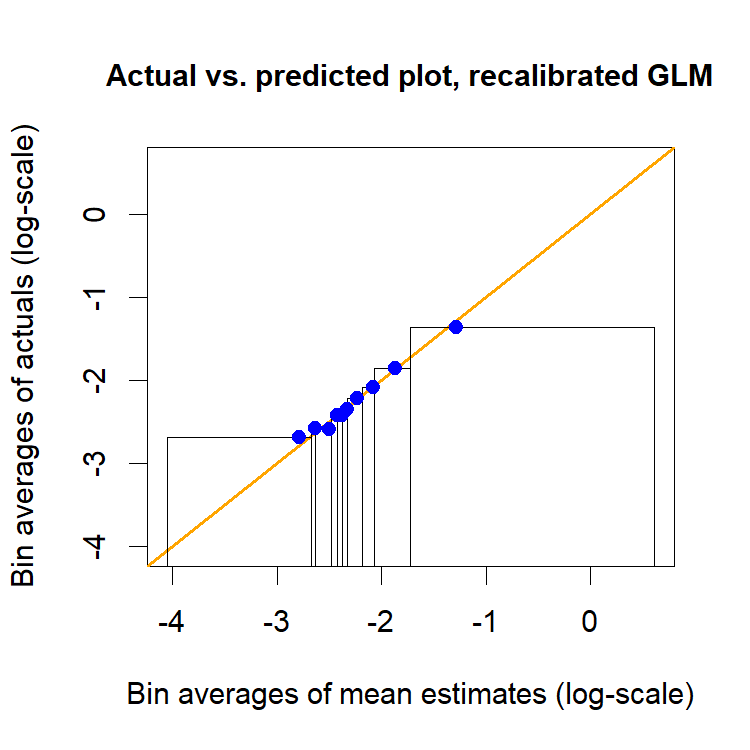}
\end{minipage}
\hspace{0.03\textwidth}
\begin{minipage}{0.42\textwidth}
\centering
\includegraphics[width=\linewidth]{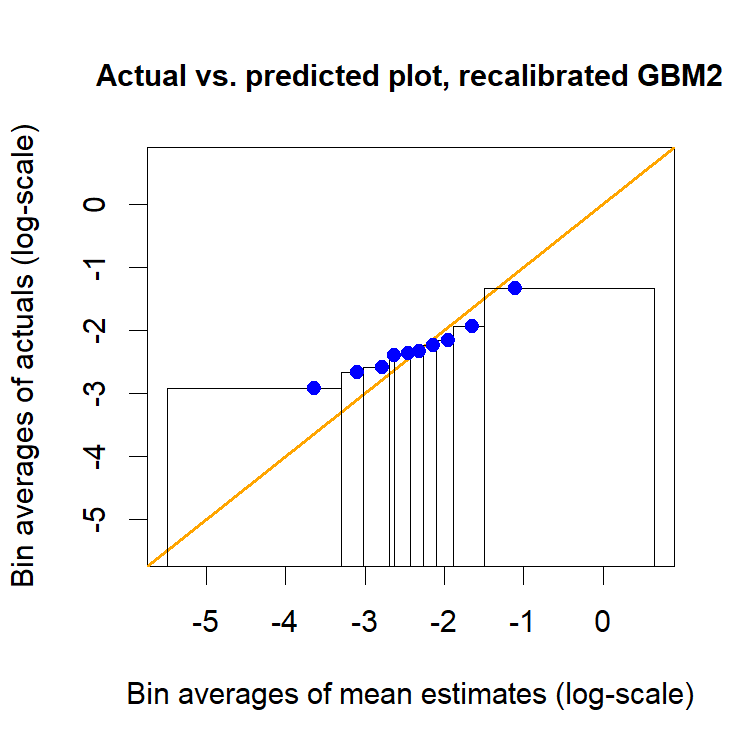}
\end{minipage}

\caption{
Plots of binned averages of responses (actuals) against binned averages of mean estimates on the test set $\T$ for two isotonic recalibrated models. Only the case where the size of the learning set corresponds to $60\%$ of the dataset $\D$ is considered. Bin intervals are shown as rectangles.
}
\label{new_Fig4}
\end{figure}

Finally, we provide the actual vs.~predicted plots for two models in Figure \ref{new_Fig4}. These plots show how isotonic recalibration helps to significantly improve the auto-calibration of the GLM, whereas it fails to do so for the model GBM2.

\subsection{Power of the tests}

In the two previous sections, we assessed the strong calibration and auto-calibration of various statistical models for the Swiss MTPL insurance dataset in Wüthrich--Buser \cite{Wüthrich--Buser}. As this dataset was synthetically generated and the true conditional mean is available for all policies, one can in principle simulate different claims histories from the same true model. We use this to evaluate the power of our tests by generating $100$ different datasets. To this end, we keep the triplets of the learning set $\l = (y_i, \bx_i, v_i)_{i \in \LL}$ fixed, while randomly reallocating the remaining policies between the boosting set $\B$ and the test set $\T$ at each simulation. In addition, the numbers of claims associated with the policies in $\B$ and $\T$ are resimulated each time. This results in $100$ new datasets of 500,000 policies, and as the learning set was fixed, we do not fit again any model, i.e., we evaluate the strong calibration and auto-calibration of the same regression functions as in Section \ref{sec considered models}.

Following the testing procedures in Sections \ref{sec num example test cal} and \ref{sec num example test auto cal}, we fit benchmark GBMs using Poisson boosting trees on $\l$ and $\B$ for each of the simulated datasets. The average number of boosting steps used to fit these models is provided in Table \ref{Tab7}, where we additionally show the average Poisson deviance losses on $\l$ and $\T$ as well as the average empirical Kullback-Leibler distances with respect to the true conditional mean attained by the benchmark GBMs. These GBMs allow us to define test functions for assessing strong calibration and auto-calibration, see Sections \ref{sec num example test cal} and \ref{sec num example test auto cal}. The number of rejections of the null-hypothesis of strong calibration at a confidence level of $1-\alpha = 0.95$ is given in Table \ref{Tab4} for our test as well as for the strong calibration test of Delong et al.~\cite{Delong}. There, we additionally show the values of the average test statistics $T_n^{\textrm{sc}}$.
\begin{table}[htb!]
  \begin{center}
 {\small   
\begin{tabular}{|l||c|c|c|c|c|}
\hline
  Models & Avg.~boost.~steps & $\bar{L}(\l, \widehat{\mu}^\textrm{boost}_{\l, \B})$ & $\bar{L}(\T, \widehat{\mu}^\textrm{boost}_{\l, \B})$ & $\overline{\textrm{KL}}_{\mu^{*}}(\T,\widehat{\mu}^\textrm{boost}_{\l, \B})$ \\
  \hline\hline
{\it Case 1.} $\mathcal L : 60\%$, $\B : 20\%$, $\T : 20\%$ & -&- &- &- \\
\hline
Benchmark GBM & 191.17 & 27.726 & 27.823 & 0.073 \\
 \hline \hline
{\it Case 2.} $\mathcal L : 80\%$, $\B : 10\%$, $\T : 10\%$ &- &- &- &-\\
\hline
Benchmark GBM & 217.06 & 27.761 & 27.837 & 0.067 \\
\hline
\end{tabular}}
\end{center}
\caption{Average number of boosting steps used to fit the benchmark GBMs for $100$ simulations. The average Poisson deviance losses and the average empirical Kullback-Leibler distances are reported in $10^{-2}$.}
\label{Tab7}
\end{table}

\begin{table}[htb!]
  \begin{center}
 {\small   
\begin{tabular}{|l||c|c|c|}
\hline
  Models 
  &\begin{tabular}{c}
   Rejections of the\\
    test of Delong et al.~\cite{Delong}
   \end{tabular}
  &\begin{tabular}{c}
   Rejections of the\\
    test of Sec.~3.3
   \end{tabular}
  &\begin{tabular}{c}
   Avg.~test \\
    stat.~$T_n^{\textrm{sc}}$
   \end{tabular}\\
\hline\hline

{\it Case 1.} $\mathcal L : 60\%$, $\B : 20\%$, $\T : 20\%$ 
& - & - & -\\
\hline
True model       & 0/100   & 2/100   & $-0.077$\\
Homogeneous mean & 24/100  & 100/100 & $23.798$\\
GLM              & 1/100   & 100/100 & $13.206$\\
GAM              & 3/100   & 100/100 & $11.783$\\
DNN1             & 4/100   & 100/100 & $8.969$\\
DNN2             & 94/100  & 100/100 & $13.901$\\
GBM1             & 0/100   & 100/100 & $7.231$\\
GBM2             & 1/100   & 100/100 & $13.926$\\
\hline
Recalibrated GLM  & 13/100  & 100/100 & $13.291$\\
Recalibrated GAM  & 6/100   & 100/100 & $11.307$\\
Recalibrated DNN1 & 1/100   & 100/100 & $7.868$\\
Recalibrated DNN2 & 56/100  & 100/100 & $11.613$\\
Recalibrated GBM1 & 64/100  & 100/100 & $8.006$\\
Recalibrated GBM2 & 100/100 & 100/100 & $18.469$\\

\hline\hline
{\it Case 2.} $\mathcal L : 80\%$, $\B : 10\%$, $\T : 10\%$ 
& - & - & -\\
\hline
True model       & 0/100   & 3/100   & $0.101$\\
Homogeneous mean & 8/100   & 100/100 & $16.926$\\
GLM              & 0/100   & 100/100 & $9.364$\\
GAM              & 2/100   & 100/100 & $8.499$\\
DNN1             & 0/100   & 100/100 & $5.812$\\
DNN2             & 27/100  & 100/100 & $7.910$\\
GBM1             & 0/100   & 100/100 & $3.228$\\
GBM2             & 3/100   & 100/100 & $8.858$\\
\hline
Recalibrated GLM  & 2/100   & 100/100 & $9.500$\\
Recalibrated GAM  & 0/100   & 100/100 & $8.122$\\
Recalibrated DNN1 & 1/100   & 100/100 & $6.000$\\
Recalibrated DNN2 & 100/100 & 100/100 & $7.638$\\
Recalibrated GBM1 & 2/100   & 100/100 & $3.792$\\
Recalibrated GBM2 & 15/100  & 100/100 & $11.626$\\
\hline
\end{tabular}}
\end{center}
\caption{Number of rejections of strong calibration for $100$ different 
simulated datasets. For each model, we provide the average value of the test statistics $T_n^{\textrm{sc}}$.}
\label{Tab4}
\end{table}

We observe in Table \ref{Tab4} that our test exhibits a much higher empirical power than the test of Delong et al.~\cite{Delong} for the considered example. Indeed, for every model other than the true conditional
mean, and hence for every model for which the null hypothesis of strong
calibration is violated, the rejection rates of our test are higher than those of the test of Delong et al.~\cite{Delong}. In fact, we even see that the only model for which the null-hypothesis of strong calibration does not systematically get rejected by our test is the true conditional mean. Table \ref{Tab4} further shows that the latter model is the only one for which the average of the test statistics is close to $0$, whereas this value is positive, as expected, for all the other models, see Section \ref{sec selecting test function}. Interestingly, the rejection rate of the true model is close to the significance level of $5 \%$ and we can see in Figure \ref{Fig:Emp Density Cal} that the empirical density of the test statistics for the true model is close to the density of a standard normal random variable for both considered cases.

\begin{figure}[htb!]
\begin{center}
\begin{minipage}[t]{0.42\textwidth}
\begin{center}
\includegraphics[width=\textwidth]{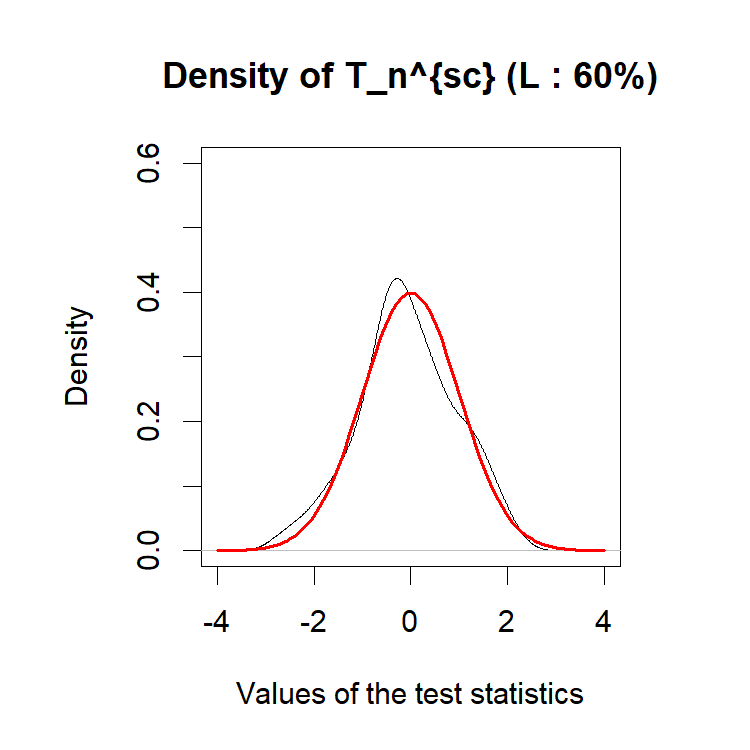}
\end{center}
\end{minipage}
\begin{minipage}[t]{0.42\textwidth}
\begin{center}
\includegraphics[width=\textwidth]{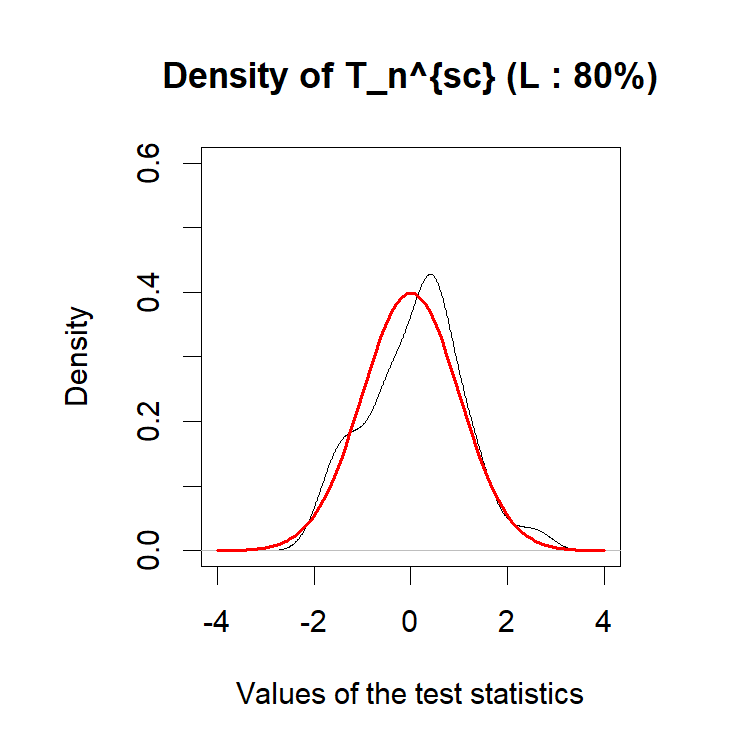}
\end{center}
\end{minipage}
\end{center}
\vspace{-.7cm}
\caption{The empirical density of the test statistic $T_n^{\textrm{sc}}$ for the true conditional mean (black) is plotted next to the theoretical density of the test statistic under the null-hypothesis of strong calibration (red).}
\label{Fig:Emp Density Cal}
\end{figure}

\begin{figure}[b!]
\begin{center}
\begin{minipage}[t]{0.42\textwidth}
\begin{center}
\includegraphics[width=\textwidth]{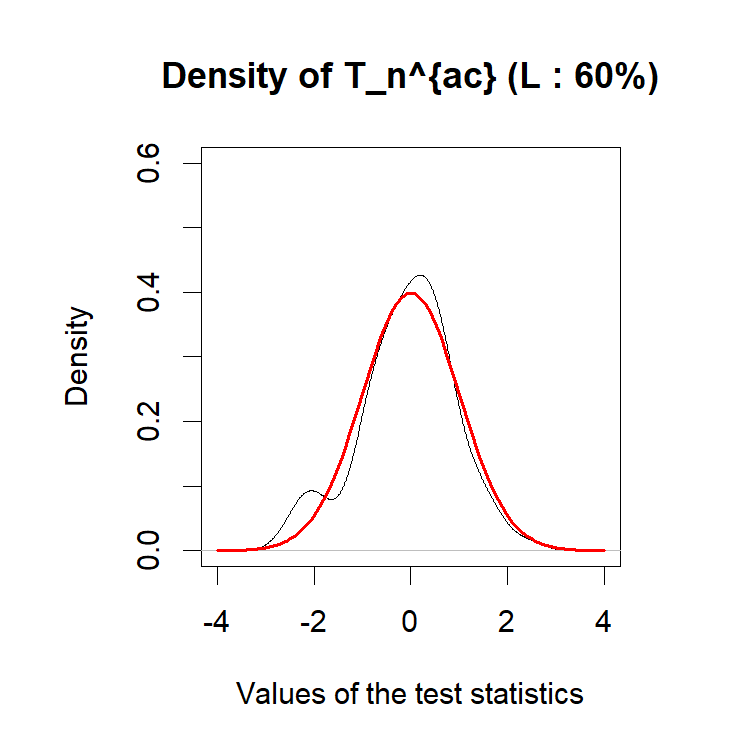}
\end{center}
\end{minipage}
\begin{minipage}[t]{0.42\textwidth}
\begin{center}
\includegraphics[width=\textwidth]{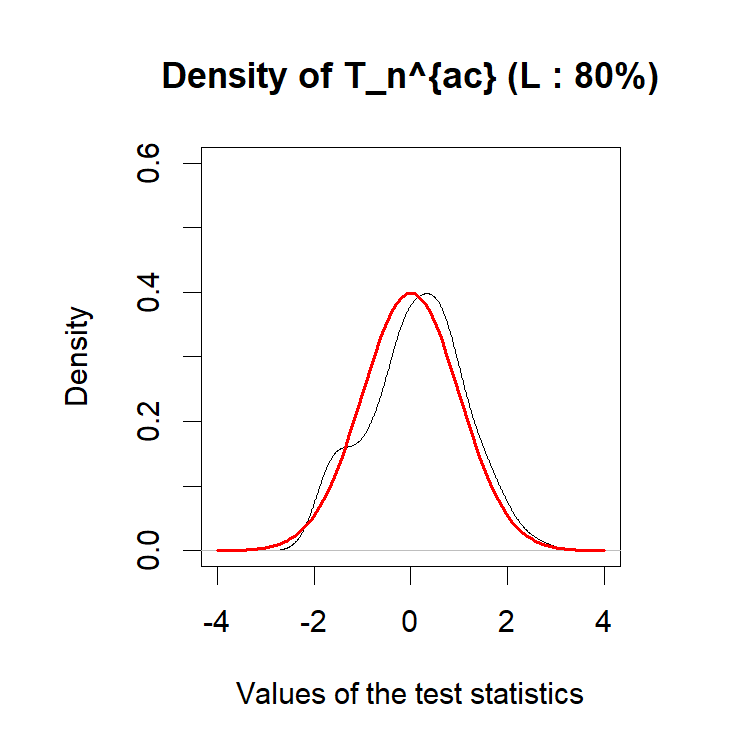}
\end{center}
\end{minipage}
\end{center}
\vspace{-.7cm}
\caption{The empirical density of the test statistic $T_n^{\textrm{ac}}$ for the true conditional mean (black) is plotted next to the theoretical density of the test statistic under the null-hypothesis of auto-calibration (red).}
\label{Fig:Emp Density Auto cal}
\end{figure}

We repeat the same procedure to assess auto-calibration at a confidence level of $1-\alpha = 0.95$. The results are given in Table \ref{Tab5}, where we see that our test exhibits a higher empirical power than the test of Denuit et al.~\cite{Denuit2}. Indeed, the rejection rates of our test are higher for all models except the true model and the homogeneous mean, with both tests achieving identical rejection rates in a few cases. In other words, the homogeneous mean is the only model violating the null hypothesis of auto-calibration for which the test of Denuit et al.~\cite{Denuit2} exhibits a higher rejection rate than our proposed test, and for all the other (more complex) models violating auto-calibration, including isotonically recalibrated models, our test turns out to be more powerful.

Focusing now on our test, we see that the true conditional mean and the homogeneous mean get rejected at a rate being close to the significance level of $5 \%$. Moreover, we notice that the null-hypothesis of auto-calibration does not always get rejected for most isotonic recalibrated models. This seems reasonable as those models are empirically auto-calibrated by construction. As for strong calibration, most average values of the test statistics $T_n^{\textrm{ac}}$ are positive in Table \ref{Tab5}, except for the homogeneous and true models. The empirical density of the test statistic for the latter model is provided in Figure \ref{Fig:Emp Density Auto cal}, and it seems again close to the density of a standard normal random variable.

\begin{table}[htb!]
  \begin{center}
 {\small   
\begin{tabular}{|l||c|c|c|}
\hline
  Models 
  &\begin{tabular}{c}
   Rejections of the\\
    test of Denuit et al.~\cite{Denuit2}
   \end{tabular}
  &\begin{tabular}{c}
   Rejections of the\\
    test of Sec.~3.4
   \end{tabular}
  &\begin{tabular}{c}
   Avg.~test \\
    stat.~$T_n^{\textrm{ac}}$
   \end{tabular}\\
   \hline \hline
{\it Case 1.} $\mathcal L : 60\%$, $\B : 20\%$, $\T : 20\%$ &- & - & -\\
\hline
\hline\hline
True model & 7/100 & 1/100 & $-0.078$ \\
Homogeneous mean & 25/100 & 4/100 & $-0.055$ \\
GLM & 29/100 & 100/100 & $10.279$ \\
GAM & 78/100 & 100/100 & $8.609$\\
DNN1 & 28/100 & 100/100 & $6.874$\\
DNN2 & 100/100 & 100/100 & $11.378$\\
GBM1 & 100/100 & 100/100 & $6.074$\\
GBM2 & 100/100 & 100/100 & $11.420$\\
\hline
Recalibrated GLM & 23/100 & 46/100 & $1.827$\\
Recalibrated GAM & 22/100 & 32/100 & $1.481$\\
Recalibrated DNN1 & 25/100 & 65/100 & $2.342$\\
Recalibrated DNN2 & 85/100 & 98/100 & $4.146$\\
Recalibrated GBM1 & 68/100 & 96/100 & $3.603$\\
Recalibrated GBM2 & 100/100 & 100/100 & $15.357$\\
\hline \hline
{\it Case 2.} $\mathcal L : 80\%$, $\B : 10\%$, $\T : 10\%$ &- & - & -\\
\hline
True model & 5/100 & 3/100 & $0.124$\\
Homogeneous mean & 9/100 & 2/100 & $-0.004$\\
GLM & 11/100 & 100/100 & $7.558$\\
GAM & 45/100 & 100/100 & $6.244$\\
DNN1 & 29/100 & 98/100 & $4.230$\\
DNN2 & 100/100 & 100/100 & $6.309$\\
GBM1 & 14/100 & 80/100 & $2.595$\\
GBM2 & 100/100 & 100/100 & $7.300$\\
\hline
Recalibrated GLM & 7/100 & 26/100 & $1.316$\\
Recalibrated GAM & 7/100 & 9/100 & $0.631$\\
Recalibrated DNN1 & 13/100 & 60/100 & $2.328$\\
Recalibrated DNN2 & 68/100 & 98/100 & $3.890$\\
Recalibrated GBM1 & 22/100 & 54/100 & $2.179$\\
Recalibrated GBM2 & 100/100 & 100/100 & $9.878$\\
\hline
\end{tabular}}
\end{center}
\caption{Number of rejections of auto-calibration for $100$ different simulated datasets. The average test statistics $T_n^{\textrm{ac}}$ are provided for each model.}
\label{Tab5}
\end{table}

\subsection{Sensitivity of the tests}

We conclude this numerical study by looking at the sensitivity of both tests for strong calibration and auto-calibration. For this, we consider five different regression functions defined through the convex linear combinations
\begin{equation}
    \label{eq delta}
    \widehat{\mu}_\l^\delta : \X \to \R, \quad \bx \mapsto  \delta \widehat{\mu}_\l^{\textrm{GLM}}(\bx) + (1-\delta) \mu^*(\bx),
\end{equation}
where $\delta \in \{0, 1/4, 1/2, 3/4, 1\}$. We then assess the strong calibration and auto-calibration of these regression functions on the same $100$ simulated datasets as above and provide the number of rejections in Figure \ref{Fig:Sensitivity}. There, we see that the strong calibration of regression functions that are different from, but close to, the true conditional mean is not systematically rejected and the same holds for auto-calibration. Moreover, the number of rejections seems to be monotonically increasing in $\delta$ as one would expect from \eqref{eq delta}.

\begin{figure}[htb!]
\begin{center}
\begin{minipage}[t]{0.42\textwidth}
\begin{center}
\includegraphics[width=\textwidth]{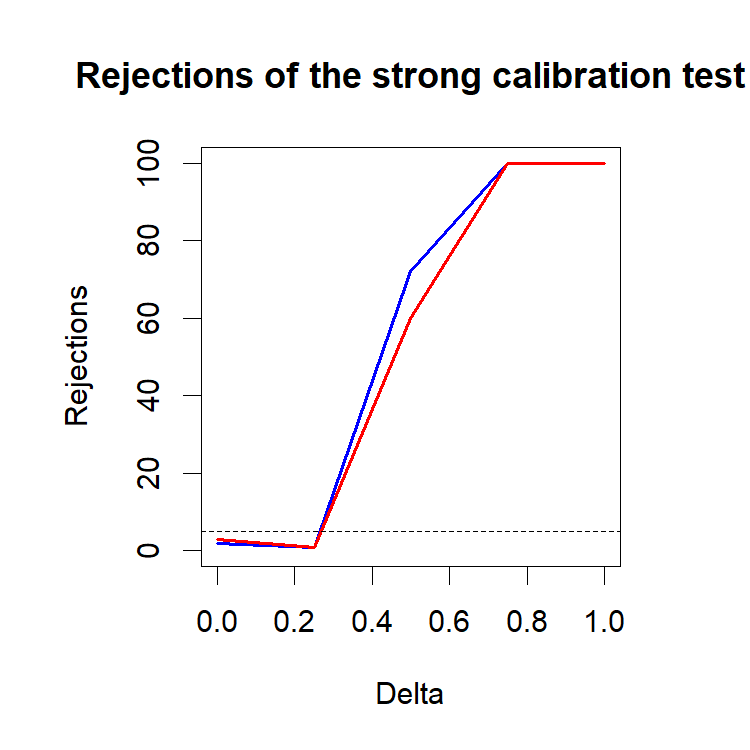}
\end{center}
\end{minipage}
\begin{minipage}[t]{0.42\textwidth}
\begin{center}
\includegraphics[width=\textwidth]{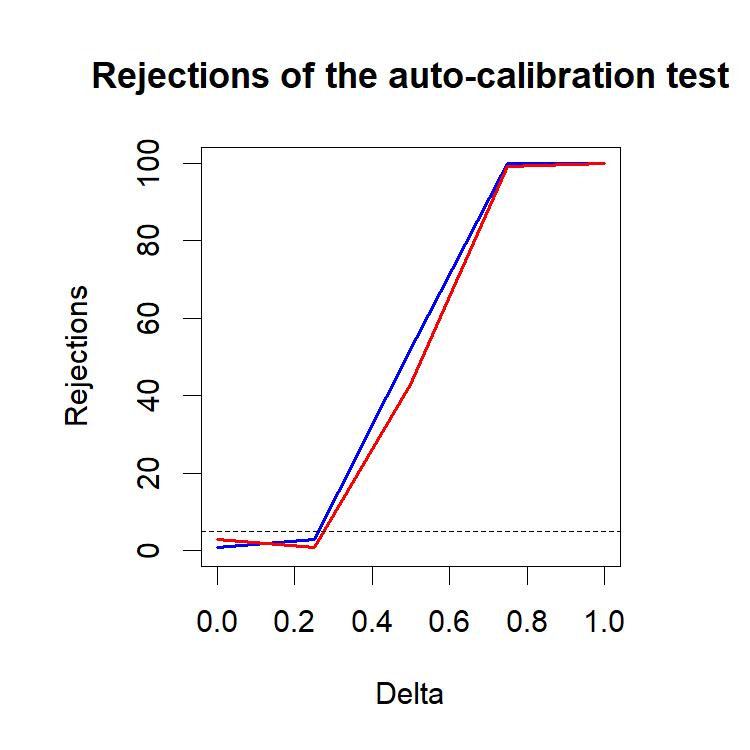}
\end{center}
\end{minipage}
\end{center}
\vspace{-.7cm}
\caption{Number of rejections of the null-hypothesis of strong calibration and auto-calibration, respectively, for $100$ different simulated datasets. The blue curves correspond to the case where the learning set corresponds to $60 \%$ of the data, whereas the red curves correspond to the other case. These are only evaluated at $\delta \in \{0, 1/4, 1/2, 3/4, 1\}$.}
\label{Fig:Sensitivity}
\end{figure}

This section, along with the previous one, shows that although we only assess a single necessary condition for strong calibration and auto-calibration, respectively, the proposed testing procedures are very powerful for the considered Swiss MTPL insurance dataset. For instance, the strong calibration of a regression function was systematically rejected whenever it was not close enough to the true conditional mean, whereas the auto-calibration of a regression function did not always get rejected for the true conditional mean, empirically auto-calibrated regression functions or regression functions being close to the true conditional mean. Moreover, we saw that for the considered example, our strong calibration test outperforms the test of Delong et al.~\cite{Delong}, whereas our auto-calibration test outperforms the test of Denuit et al.~\cite{Denuit2}. 

Interestingly, the rejection rates of our tests for the true model, which is the only model for which the null-hypotheses of strong calibration and auto-calibration are not violated, were close to the significance level of $5 \%$ in this example, and the convergence results in Sections \ref{sec assessing calibration} and \ref{sec assessing auto calibration} were shown to hold in Figures \ref{Fig:Emp Density Cal} and \ref{Fig:Emp Density Auto cal} for the considered sizes of the test set $\T$. Again, we emphasize that we do not claim that our tests are uniformly more powerful than the above competitors, especially as they only assess necessary conditions. However, the considered example shows that they can achieve high power if the benchmark GBM successfully identifies violations of strong calibration and auto-calibration. Finally, note that a further numerical study is available in Appendix \ref{appendix real dataset}, where we apply our proposed testing procedures to the French MTPL real dataset available in the
{\sf R} \cite{R Core Team} package {\tt CASdatasets} hosted by Dutang--Charpentier \cite{Dutang--Charpentier}. The results of this study are largely consistent with those obtained in Sections \ref{sec num example test cal} and \ref{sec num example test auto cal}.

\section{Conclusion}

This paper introduces new testing procedures to assess the strong calibration and auto-calibration of fitted regression functions. To this end, equivalent characterizations of these properties arising from the measure-theoretic definition of conditional expectation are used to derive necessary conditions. Although testing for necessary conditions generally leads to low statistical power, we show in this paper that boosting trees enable us to construct statistical tests based on a single necessary condition that are able to detect various kinds of violations of strong calibration and auto-calibration, respectively. This is supported by a numerical example, in which the proposed tests prove to be very powerful. In particular, we show that, in this example, our strong calibration test outperforms the test of Delong et al.~\cite{Delong}, whereas our auto-calibration test outperforms the test of Denuit et al.~\cite{Denuit2}.

The selection of the single necessary condition to be assessed plays an important role in the power of our tests and our approach consists in fitting a benchmark gradient boosting model on the learning set and an additional set, that we call boosting set, in order to learn violations of strong calibration and auto-calibration for the originally fitted regression function. The use of boosting is motivated by the strong out-of-sample predictive performance achieved by such models. Going forward, it might be interesting to consider the use of other statistical models for this purpose as, for example, deep neural networks. Moreover, the selection of the necessary conditions to be assessed could in principle be done in other ways, but as pointed out in this paper, it seems important to use a method that adapts to the regression function under consideration in order to obtain powerful tests.

Finally, we emphasize that the idea of using a subpart of the dataset that was not used in the fitting stage to first identify violations of strong calibration or auto-calibration and, then, select a suitable test statistic could be used for other testing procedures as well. This is, for example, the case for the auto-calibration test of Denuit et al.~\cite{Denuit2}, which is based on a Kolmogorov-Smirnov type test statistic. Focusing on parts of the portfolio where violations of the null-hypothesis are most likely to occur could improve the power of such tests.

\bigskip

{\bf Acknowledgments.} The author thanks Mario Wüthrich for the valuable feedback and the useful remarks received on earlier versions of this work, as well as the anonymous referees and the editor for their constructive comments and suggestions that have helped improve the manuscript.

\bigskip

{\small %\baselineskip.5em
\renewcommand{\baselinestretch}{.51}
}

\appendix

%\section{Parameter estimation and initialization}

\section{Proofs}
We prove all statements in this appendix, except for Proposition \ref{prop equiv auto-cal}. Its proof is analogous to the proof of Proposition \ref{prop equiv cal} and is therefore omitted.

\bigskip
{\Beweis {\bf Proof of Proposition \ref{prop equiv cal}.}
Suppose first that the regression function $\widehat{\mu} : \X \to \R$ is strongly calibrated for $(Y,\bX)$ and let $g:\X\to\R$ be a measurable function such that $\left(Y-\widehat{\mu}(\bX)\right)g(\bX) \in L^1(\p)$. Then, define the function
\[
    g_K(\bx)
    =
    \max\{-K,\min\{g(\bx),K\}\},
\]
for $K \in \mathbb{N}$. Since $g_K(\bX)$ is bounded and $Y, \widehat{\mu}(\bX) \in L^1(\p)$, one can write using the tower
property
\begin{equation*}
    \E\left[\left(Y-\widehat{\mu}(\bX)\right)g_K(\bX)\right]
    =
    \E\left[\E\left[\left(Y-\widehat{\mu}(\bX)\right)g_K(\bX)\mid \bX \right]\right] 
    = \E\left[\left(\E\left[Y \mid \bX]-\widehat{\mu}(\bX)\right) g_K(\bX)\right] \right] = 0.
\end{equation*}
Moreover, as 
\[
    \left|\left(Y-\widehat{\mu}(\bX)\right)g_K(\bX)\right|\leq \left|\left(Y-\widehat{\mu}(\bX)\right)g(\bX)\right| \quad \textrm{and} \quad \left(Y-\widehat{\mu}(\bX)\right)g_K(\bX) \stackrel{\p\textrm{-a.s.}}{\longrightarrow} \left(Y-\widehat{\mu}(\bX)\right)g(\bX),
\]
as $K \to \infty$, the dominated convergence theorem allows us to conclude that
\[
    \E[\left(Y-\widehat{\mu}(\bX)\right)g(\bX)]
    =
    \lim_{K\to\infty}\E[\left(Y-\widehat{\mu}(\bX)\right)g_K(\bX)]
    =0,
\]
showing \eqref{eq prop cal}. Conversely, suppose that \eqref{eq prop cal} holds for every
measurable function $g:\X\to\R$ such that $\left(Y-\widehat{\mu}(\bX)\right)g(\bX) \in L^1(\p)$. Then, for any measurable set $S \in \sigma(\bX)$, it suffices to set $g = \mathds{1}_S$ in \eqref{eq prop cal} to obtain
\[
    \E\left[Y\mathds{1}_S\right]
    =
    \E\left[\widehat{\mu}(\bX)
    \mathds{1}_S\right].
\]
This shows that the regression function $\widehat{\mu} : \X \to \R$ is strongly calibrated for $(Y,\bX)$, see \eqref{cond exp calibration}.
\bigskip
\EndProof}

{\Beweis {\bf Proof of Proposition \ref{prop theory cal}.}
    The random variables $(Z_i)_{i=1}^n$ in \eqref{prop theory cal Z_i} are conditionally i.i.d.~given $\B$ because the pairs $(Y_i, \bX_i)_{i \in \TT}$ were assumed to be i.i.d.~and $\B$ is independent of $\T$. Moreover, under the null-hypothesis of strong calibration of $\widehat{\mu}_\ell : \X \to \R$, we have
    $$
        \E[(Y - \widehat{\mu}_\ell(\bX)) g_{\ell, \B}(\bX) \, | \, \B] =0, \quad \p\textrm{-a.s.}
    $$
    Define
    $$
    \sigma^2(\B) = \textrm{Var}\left[(Y-\mu(\bX)) g_{\ell, \B}(\bX) \,| \, \B \right] > 0, \quad \p\textrm{-a.s.}
    $$
    Then, for almost every realization $\B = \bb$, we obtain using the central limit theorem (CLT),
    \begin{equation}
        \label{eq proof 5}
        \left.\sqrt{n} \bar{Z} \, \right| \, \B = \bb \stackrel{d}{\longrightarrow} \mathcal{N}(0, \sigma^2(\bb)),
    \end{equation}
    as $n \to \infty$. Since $S_Z^2$ is a consistent estimator of the variance $\sigma^2(\bb)$, we further have
    \begin{equation}
    \label{eq proof 6}
    S_Z^2 \stackrel{p}{\longrightarrow} \sigma^2(\bb) > 0,
    \end{equation}
    as $n \to \infty$ by the law of large numbers. Finally, the continuous mapping theorem and Slutsky's theorem allow us to conclude from \eqref{eq proof 5} and \eqref{eq proof 6} that
    \begin{equation}
        \label{limit cond V}
        T_n^{{\textrm{sc}}} = \left.\frac{\bar{Z}}{\sqrt{S_Z^2/n}} \,\right| \, \B= \bb \stackrel{d}{\longrightarrow} \mathcal{N}(0,1),
    \end{equation}
    as $n \to \infty$. As the above result holds for almost every realization $\B = \bb$ and the limiting distribution in \eqref{limit cond V} does not depend on $\bb$, we finally conclude that
    $$
    T_n^{{\textrm{sc}}} \stackrel{d}{\longrightarrow} \mathcal{N}(0,1),
    $$
    as $n \to \infty$. This shows the claim.
\bigskip
\EndProof}

\section{Comparison with the strong calibration test of Delong et al.}

\label{appendix cal bands}

\begin{figure}[b!]
\begin{center}
\begin{minipage}[t]{0.42\textwidth}
\begin{center}
\includegraphics[width=\textwidth]{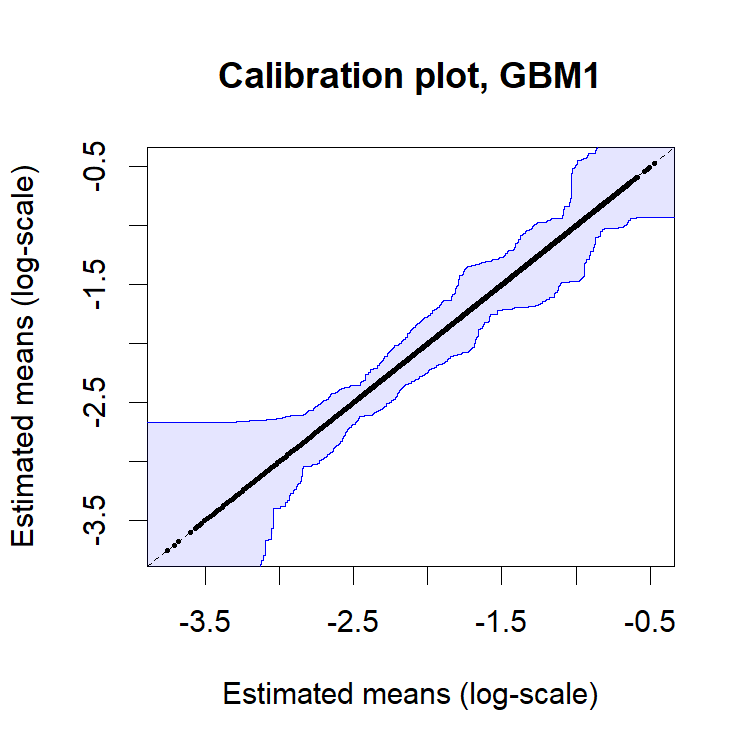}
\end{center}
\end{minipage}
\begin{minipage}[t]{0.42\textwidth}
\begin{center}
\includegraphics[width=\textwidth]{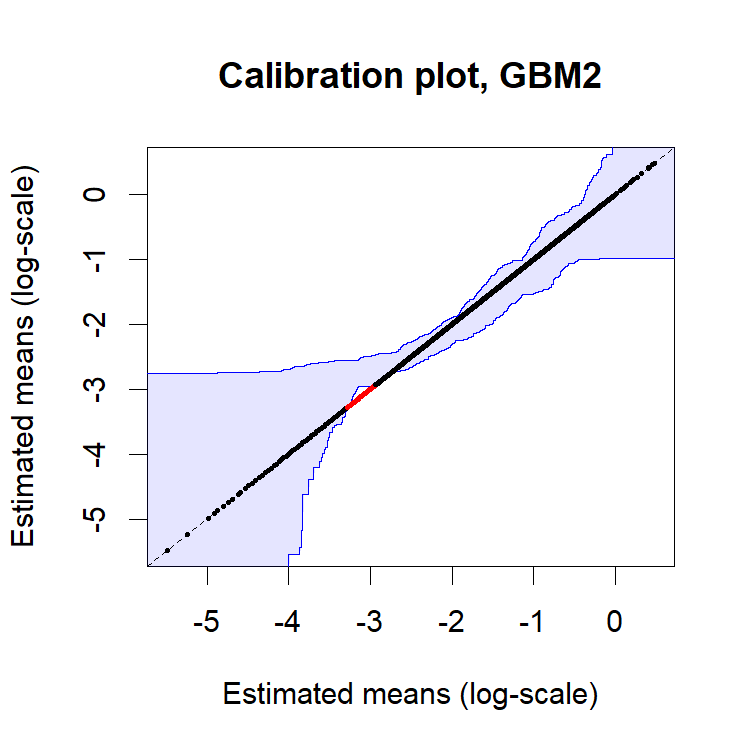}
\end{center}
\end{minipage}
\end{center}
\vspace{-.7cm}
\caption{Calibration plots of two different models. Only the case where the learning set corresponds to 60\% of the dataset $\D$ is considered. The mean estimates are plotted on the diagonal. The points falling within the band are drawn in black, whereas those falling outside the bands are drawn in red.}
\label{fig:Cal band 1}
\end{figure}

In this appendix, we apply the strong calibration test of Delong et al.~\cite{Delong} to the models introduced in Section \ref{sec considered models}, and compare its conclusions with the results obtained in Section \ref{sec num example test cal}. This test is based on calibration bands and works as follows. To assess strong calibration, one has to plot the mean estimates provided by the regression function $\widehat{\mu}_\l : \X \to \R$ on the diagonal and reject the strong calibration of this regression function whenever some mean estimates fall outside the calibration band. These authors call such a plot a calibration plot. We construct calibration bands for all considered models at a confidence level of $1-\alpha = 0.95$ and show two of the resulting plots in Figure \ref{fig:Cal band 1}. As this method does not involve any boosting set $\B$, we emphasize that all bands were constructed using the data $\D \setminus \mathcal{L}$. Moreover, note that responses were binned according to their mean estimates as in Section 8.4 of Delong et al.~\cite{Delong}.

Figure \ref{fig:Cal band 1} shows that the strong calibration of the model GBM1 is not rejected by the test of Delong et al.~\cite{Delong}, whereas the strong calibration of the model GBM2 is rejected because some points lie outside the band. A summary of the results induced by this test is provided for all models in Table \ref{Tab_new_appendix}. There, we observe that the constructed calibration bands are only able to detect violations of strong calibration for the overfitting models DNN2 and GBM2 in the case where the learning set corresponds to $60 \%$ of the dataset $\D$. This can be explained by two different factors. On the one hand, calibration bands become narrower as more data is used to construct them, making violations of strong calibration easier to detect when the test set represents 40\% of the data, and, on the other hand, these models are among those exhibiting the largest empirical Kullback-Leibler distance with respect to the true conditional mean, see Table \ref{Tab1}. As responses are binned according to the underlying mean estimates in the construction of the bands, this method is not suitable to assess strong calibration of models that are much less granular than the true conditional mean. It is thus not surprising that the strong calibration of the homogeneous model does not get rejected by this test.

\begin{table}[htb!]
  \begin{center}
 {\small   
\begin{tabular}{|l||c|}
\hline
  Models & \begin{tabular}{c}
 Results of the \\
 test of Delong et al.~\cite{Delong}
\end{tabular}\\
  \hline\hline
{\it Case 1.} $\mathcal L : 60\%$, $\B : 20\%$, $\T : 20\%$ &- \\
\hline
True model &\ding{51}\\
Homogeneous mean & \ding{51}\\
GLM & \ding{51}\\
GAM & \ding{51}\\
DNN1 & \ding{51}\\
DNN2 & \ding{55}\\
GBM1 &\ding{51}\\
GBM2 &\ding{55}\\

 \hline \hline
{\it Case 2.} $\mathcal L : 80\%$, $\B : 10\%$, $\T : 10\%$ &- \\
\hline
True model &\ding{51}\\
Homogeneous mean &\ding{51}\\
GLM &\ding{51}\\
GAM &\ding{51}\\
DNN1 &\ding{51}\\
DNN2 &\ding{51}\\
GBM1 &\ding{51}\\
GBM2 &\ding{51}\\
\hline
\end{tabular}}
\end{center}
\caption{Results of the strong calibration test of Delong et al.~\cite{Delong}. A rejection of the null hypothesis of strong calibration is indicated by \ding{55}, whereas a non-rejection is indicated by \ding{51}.}
\label{Tab_new_appendix}
\end{table}

\section{Assessing model calibration on a real insurance dataset}

\label{appendix real dataset}

In this appendix, we apply the strong calibration and auto-calibration tests introduced in Sections \ref{sec assessing calibration} and \ref{sec assessing auto calibration} to the French MTPL dataset available in the
{\sf R} \cite{R Core Team} package {\tt CASdatasets} hosted by Dutang--Charpentier \cite{Dutang--Charpentier}\footnote{We use here a cleaned version of the dataset, which can be downloaded under https://people.math.ethz.ch/wueth/Lecture/freMTPL2freq.rda. Details on the cleaning procedure can be found in Listing 13.1 of Wüthrich--Merz \cite{WM2023}.}. In contrast to the Swiss synthetic dataset studied in Section \ref{sec num example}, this dataset consists of real insurance policies. It includes $n =$ 678,007 French car drivers and, for each policy $1 \leq i \leq n$, it provides the number of claims $N_i \in \N$ that occurred during an exposure period $v_i \in (0,1]$ together with features containing information on the policyholders. As the Swiss dataset was synthetically constructed based on this French dataset, most features have direct counterparts in both datasets. Table \ref{Tab:correspendence} provides these counterparts and shows that the available information on the policyholders is largely similar, with the exception of the feature ${\tt BonusMalus}$ that is only available in the French dataset. Further details on the latter dataset can be found in Section 13.1 in Wüthrich--Merz \cite{WM2023}.

\begin{table}[htb!]
\begin{center}
{\small
\begin{tabular}{|l|l|}
\hline
French MTPL & Swiss MTPL \\
\hline\hline
{\tt DrivAge}     & {\tt age} \\
{\tt VehAge}      & {\tt ac} \\
{\tt VehPower}    & {\tt power} \\
{\tt VehBrand}    & {\tt brand} \\
{\tt VehGas}      & {\tt gas} \\
{\tt Area}        & {\tt area} \\
{\tt Density}     & {\tt dens} \\
{\tt Region}      & {\tt ct} \\
{\tt BonusMalus}  & {\tt-} \\
\hline
\end{tabular}
}
\end{center}
\caption{Correspondence between the features of the French MTPL real dataset and the Swiss MTPL synthetic dataset. The feature {\tt BonusMalus} is only available in the French dataset.}
\label{Tab:correspendence}
\end{table}

We follow the approach described in Section \ref{sec considered models} to fit seven regression functions aiming to approximate the true annual frequency of claims of the policyholders. Each of these regression functions is fitted using exactly the same methodology and features as for the Swiss dataset, up to the differences in feature names reported in Table \ref{Tab:correspendence} and with the addition of the feature {\tt BonusMalus} in the feature space. We then perform the strong calibration and auto-calibration tests of Sections \ref{sec num example test cal} and \ref{sec num example test auto cal} by fitting a benchmark GBM. Again, this benchmark GBM is fitted exactly as in Section \ref{sec num example test cal} with the addition of the feature {\tt BonusMalus}. Moreover, note that, as in Section \ref{sec num example}, we consider two different cases for the size of the learning, boosting and test sets.

The results of the performed tests at a confidence level of \(1-\alpha=0.95\) are available in Tables \ref{Tab:French sc} and \ref{Tab:French ac}, where we compare them to the results of the strong calibration test of Delong et al.~\cite{Delong} and the auto-calibration test of Denuit et al.~\cite{Denuit2}, respectively. On the one hand, our strong calibration test rejects the null-hypothesis of strong calibration for all considered models, whereas the test of Delong et al.~\cite{Delong} only rejects it for the overfitting DNN in the case where the learning set corresponds to $60 \%$ of the data. On the other hand, our auto-calibration test rejects the null hypothesis of auto-calibration for all models except the homogeneous mean and some of the isotonic recalibrated models. The test of Denuit et al.~\cite{Denuit2}, for its part, rejects auto-calibration for several models, but for fewer models than our proposed test. Interestingly, all models for which their test rejects the null hypothesis of auto-calibration are also rejected by ours. These results indicate that our proposed tests can detect violations of strong calibration and auto-calibration that are not identified by the competing procedures for the considered example.

\vspace*{\fill}
\begin{table}[htb!]
  \begin{center}
{\small
\begin{tabular}{|l||c|c|}
\hline
  Models
  &\begin{tabular}{c}
   Results of the\\
   test of Delong et al.~\cite{Delong}
   \end{tabular}
  &\begin{tabular}{c}
   $p$-values of the\\
   test of Sec.~\ref{sec assessing calibration}
   \end{tabular}\\
\hline\hline

{\it Case 1.} $\mathcal L : 60\%$, $\B : 20\%$, $\T : 20\%$
& - & -\\
\hline
Homogeneous mean & \ding{51} & $3.68 \cdot 10^{-244}$\\
GLM              & \ding{51} & $1.99 \cdot 10^{-73}$\\
GAM              & \ding{51} & $3.16 \cdot 10^{-62}$\\
DNN1             & \ding{51} & $5.82 \cdot 10^{-43}$\\
DNN2             & \ding{55} & $7.22 \cdot 10^{-70}$\\
GBM1             & \ding{51} & $9.36 \cdot 10^{-16}$\\
GBM2             & \ding{51} & $6.64 \cdot 10^{-54}$\\

\hline\hline
{\it Case 2.} $\mathcal L : 80\%$, $\B : 10\%$, $\T : 10\%$
& - & -\\
\hline
Homogeneous mean & \ding{51} & $2.59 \cdot 10^{-122}$\\
GLM              & \ding{51} & $2.32 \cdot 10^{-33}$\\
GAM              & \ding{51} & $1.37 \cdot 10^{-25}$\\
DNN1             & \ding{51} & $2.01 \cdot 10^{-18}$\\
DNN2             & \ding{51} & $1.57 \cdot 10^{-33}$\\
GBM1             & \ding{51} & $1.98 \cdot 10^{-9}$\\
GBM2             & \ding{51} & $3.28 \cdot 10^{-27}$\\
\hline
\end{tabular}}
\end{center}
\caption{$p$-values of the strong calibration test of Section \ref{sec assessing calibration} for seven different models and two different cases together with the results of the strong calibration test of Delong et al.~\cite{Delong}. For the latter test, a rejection of the null hypothesis of strong calibration is indicated by \ding{55}, whereas a non-rejection is indicated by \ding{51}.}
\label{Tab:French sc}
\end{table}
\vspace*{\fill}

\begin{table}[htb!]
  \begin{center}
 {\small
\begin{tabular}{|l||c|c|}
\hline
  Models
  &\begin{tabular}{c}
   $p$-values of the\\
   test of Denuit et al.~\cite{Denuit2}
   \end{tabular}
  &\begin{tabular}{c}
   $p$-values of the\\
   test of Sec.~\ref{sec assessing auto calibration}
   \end{tabular}\\
\hline\hline

{\it Case 1.} $\mathcal L : 60\%$, $\B : 20\%$, $\T : 20\%$
& - & -\\
\hline
Homogeneous mean & $5.32 \cdot 10^{-1}$ & $8.56 \cdot 10^{-1}$\\
GLM              & $3.20 \cdot 10^{-2}$ & $1.01 \cdot 10^{-32}$\\
GAM              & $5.80 \cdot 10^{-2}$ & $2.56 \cdot 10^{-6}$\\
DNN1             & $1.98 \cdot 10^{-1}$ & $4.28 \cdot 10^{-3}$\\
DNN2             & $0$                  & $4.68 \cdot 10^{-7}$\\
GBM1             & $1.12 \cdot 10^{-1}$ & $3.06 \cdot 10^{-4}$\\
GBM2             & $0$                  & $2.67 \cdot 10^{-9}$\\
\hline
Recalibrated GLM  & $7.18 \cdot 10^{-1}$ & $2.26 \cdot 10^{-1}$\\
Recalibrated GAM  & $9.00 \cdot 10^{-1}$ & $5.37 \cdot 10^{-1}$\\
Recalibrated DNN1 & $8.80 \cdot 10^{-2}$ & $2.44 \cdot 10^{-4}$\\
Recalibrated DNN2 & 0                    & $1.18 \cdot 10^{-6}$\\
Recalibrated GBM1 & 0                    & $4.60 \cdot 10^{-7}$\\
Recalibrated GBM2 & 0                    & $2.64 \cdot 10^{-43}$\\

\hline\hline
{\it Case 2.} $\mathcal L : 80\%$, $\B : 10\%$, $\T : 10\%$
& - & -\\
\hline
Homogeneous mean & $6.58 \cdot 10^{-1}$ & $5.60 \cdot 10^{-1}$\\
GLM              & $1.18 \cdot 10^{-1}$ & $1.76 \cdot 10^{-17}$\\
GAM              & $3.10 \cdot 10^{-1}$ & $2.52 \cdot 10^{-5}$\\
DNN1             & $0$                  & $2.85 \cdot 10^{-4}$\\
DNN2             & $0$                  & $1.14 \cdot 10^{-3}$\\
GBM1             & $2.44 \cdot 10^{-1}$ & $1.19 \cdot 10^{-2}$\\
GBM2             & $0$                  & $1.45 \cdot 10^{-3}$\\
\hline
Recalibrated GLM  & $4.42 \cdot 10^{-1}$ & $6.30 \cdot 10^{-1}$\\
Recalibrated GAM  & $2.98 \cdot 10^{-1}$ & $2.21 \cdot 10^{-1}$\\
Recalibrated DNN1 & $1.06 \cdot 10^{-1}$ & $2.42 \cdot 10^{-1}$\\
Recalibrated DNN2 & 0                    & $1.06 \cdot 10^{-3}$\\
Recalibrated GBM1 & $1.20 \cdot 10^{-2}$ & $1.67 \cdot 10^{-2}$\\
Recalibrated GBM2 & 0                    & $4.54 \cdot 10^{-18}$\\
\hline
\end{tabular}}
\end{center}
\caption{$p$-values of the auto-calibration tests for seven different models and two different cases.}
\label{Tab:French ac}
\end{table}
\end{document}